\documentclass[a4paper, 10pt, parskip=half]{scrartcl}

\usepackage[utf8]{inputenc}
\usepackage[T1]{fontenc}
\usepackage{lmodern}
\usepackage{booktabs}
\usepackage{amsmath}
\usepackage{amssymb}
\usepackage{amsthm}
\usepackage{amsfonts}
\usepackage{dsfont}
\usepackage{mathtools}
\usepackage{bbm}
\usepackage{hyperref}
\hypersetup{
  colorlinks=true,
  linkcolor=black,
  citecolor=black,
  urlcolor=blue,
  pdftitle={Relaxed parameter conditions for chemotactic collapse in logistic-type parabolic--elliptic Keller--Segel systems},
  pdfauthor={Tobias Black, Mario Fuest, Johannes Lankeit},
  pdfkeywords={chemotaxis; finite-time blow-up; nonlinear diffusion logistic source},
  bookmarksopen=true,
}

\usepackage[normalem]{ulem}

\usepackage[numbers, sort&compress]{natbib}
\RequirePackage{geometry}
\newcommand{\R}{\mathbb{R}}

\newcommand{\ur}[1]{\mathrm{#1}}
\newcommand{\ure}{\ur e}

\ifdefined\labelenumi
  \renewcommand{\labelenumi}{(\roman{enumi})}
  
\fi

\newcommand{\eps}{\varepsilon}

\newcommand{\gt}{>}
\newcommand{\lt}{<}

\newcommand{\defs}{\coloneqq}
\newcommand{\sfed}{\eqqcolon}

\newcommand{\ra}{\rightarrow}

\newcommand{\nea}{\nearrow}

\newcommand{\sea}{\searrow}

\newcommand{\ol}{\overline}

\newcommand{\dx}{\,\mathrm{d}x}
\newcommand{\ds}{\,\mathrm{d}s}

\newcommand{\dr}{\,\mathrm{d}r}

\newcommand{\dsigma}{\,\mathrm{d}\sigma}

\newcommand{\drho}{\,\mathrm{d}\rho}

\newcommand{\ddt}{\frac{\mathrm{d}}{\mathrm{d}t}}

\newcommand{\hp}{\hphantom}
\newcommand{\pe}{\mathrel{\hp{=}}}

\newcommand{\set}[1]{\{#1\}}
\newcommand{\tmax}{T_{\max}}

\newcommand{\intom}{\int_\Omega}
\newcommand{\io}{\intom}
\newcommand{\intns}{\int_0^{s_0}}
\newcommand{\ombar}{\ol \Omega}

\newcommand{\leb}[1]{{L^{#1}(\Omega)}}

\newcommand{\con}[1]{{C^{#1}(\ombar)}}

\newcommand{\kappape}{\kappa_{\mathrm{pe}}}

\newcommand{\kl}[1]{\left(#1\right)}
\newcommand{\nn}{\nonumber}
\newcommand{\Ombar}{\overline{\Omega}}
\newcommand{\Tmax}{\tmax}
\newcommand{\f}[2]{\frac{#1}{#2}}
\usepackage{newunicodechar}
\newunicodechar{λ}{\lambda}
\newunicodechar{μ}{\mu}
\newunicodechar{ϕ}{\phi}
\newunicodechar{ψ}{\psi}
\newunicodechar{κ}{\kappa}
\newunicodechar{α}{\alpha}
\newunicodechar{ρ}{\rho}
\newunicodechar{γ}{\gamma}
\newunicodechar{ℝ}{\mathbb{R}}
\newunicodechar{∇}{\nabla}
\newunicodechar{Δ}{\Delta} 
\newunicodechar{φ}{\varphi}
\newunicodechar{η}{\eta}
\newunicodechar{θ}{\theta}
\newunicodechar{σ}{\sigma}
\newunicodechar{β}{\beta}
\newunicodechar{∞}{\infty}
\newunicodechar{ω}{\omega}

\newcommand{\Beta}{B}

\let\originalparagraph\paragraph
\renewcommand{\paragraph}[2][.]{\originalparagraph{#2#1}}

\makeatletter
\renewenvironment{proof}[1][\proofname]{\par
  \pushQED{\qed}%
  \normalfont \topsep0\p@\relax
  \trivlist
  \item[\hskip\labelsep\scshape
  #1\@addpunct{.}]\ignorespaces
}{%
  \popQED\endtrivlist\@endpefalse
}
\makeatother

\newtheorem{base}{Base}[section]
\numberwithin{equation}{section}

\theoremstyle{nplain}
\newtheorem{theorem}[base]{Theorem} \newtheorem*{theorem*}{Theroem}
\newtheorem{lemma}[base]{Lemma} \newtheorem*{lemma*}{Lemma}
 \newtheorem*{prop*}{Proposition}
 \newtheorem*{cor*}{Corollary}

\theoremstyle{ndefinition}
 \newtheorem*{definition*}{Definition}
 \newtheorem*{example*}{Example}
 \newtheorem*{cond*}{Condition}

\newtheorem{remark}[base]{Remark} \newtheorem*{remark*}{Remark}

\setkomafont{title}{\normalfont \Large}
\title{Relaxed parameter conditions for chemotactic collapse in logistic-type parabolic--elliptic Keller--Segel systems}

\usepackage{authblk}

\author[1]{Tobias Black\footnote{e-mail: tblack@math.upb.de}}
\author[1]{Mario Fuest\footnote{e-mail: fuestm@math.upb.de, corresponding author}}
\author[2,1]{Johannes Lankeit\footnote{e-mail: lankeit@ifam.uni-hannover.de}}
\affil[1]{Institut für Mathematik, Universität Paderborn, Warburger Str.~100, 33098 Paderborn, Germany}  
\affil[2]{Leibniz Universität Hannover, Institut für Angewandte Mathematik, Welfengarten 1, 30167 Hannover, Germany}
\date{}

\begin{document}
\maketitle

\KOMAoptions{abstract=true}
\begin{abstract}
  \noindent
We study the finite-time blow-up in two variants of the  parabolic--elliptic Keller--Segel system with nonlinear diffusion and logistic source. In $n$-dimensional balls, we consider
\begin{align*}
  \begin{cases}
    u_t =  \nabla \cdot ((u+1)^{m-1}\nabla u -u \nabla v) + \lambda u - \mu u^{1+\kappa}, \\
    0   = \Delta v - \frac1{|\Omega|} \int_\Omega u + u
  \end{cases} \tag{JL}
\end{align*}  
and
\begin{align*}
  \begin{cases}
    u_t = \nabla \cdot ((u+1)^{m-1}\nabla u -u \nabla v) + \lambda u - \mu u^{1+\kappa}, \\
    0   = \Delta v - v + u,
  \end{cases}\tag{PE}
\end{align*}
where $\lambda$ and $\mu$ are given spatially radial nonnegative functions and $m, \kappa > 0$ are given parameters subject to further conditions.\\[0.5em]
In a unified treatment, we establish a bridge between previously employed methods on blow-up detection and relatively new results on pointwise upper estimates of solutions in both of the systems above and then, making use of this newly found connection, provide extended parameter ranges for $m,\kappa$ leading to the existence of finite-time blow-up solutions in space dimensions three and above.\\[0.5em]
In particular, for constant $\lambda, \mu > 0$, we find that there are initial data which lead to blow-up in (JL) if
     \begin{alignat*}{2}
      0 \leq \kappa &< \min\left\{\frac{1}{2}, \frac{n - 2}{n} - (m-1)_+ \right\}&&\qquad\text{if } m\in\left[\frac{2}{n},\frac{2n-2}{n}\right)\\
      \text{ or }\quad 
      0 \leq \kappa&<\min\left\{\frac{1}{2},\frac{n-1}n-\frac{m}2\right\} &&\qquad \text{if } m\in\left(0,\frac{2}{n}\right),     \end{alignat*}
and in (PE) if $m \in [1, \frac{2n-2}{n})$ and
  \begin{align*}
      0 \leq \kappa <  \min\left\{\frac{(m-1) n + 1}{2(n-1)}, \frac{n - 2 - (m-1) n}{n(n-1)} \right\}.
    \end{align*}
\end{abstract}

\section{Introduction}
How strong does a degrading term need to be in order to rule out chemotactic collapse in a Keller--Segel system?
Or, phrased differently, when is chemotactic aggregation stronger than even superlinear dampening?

On the one hand, in the absence of any degrading terms the minimal Keller--Segel model,
proposed in the 1970s \cite{KellerSegelTravelingBandsChemotactic1971} to model chemotaxis, that is,
the directed movement of, for instance, cells or bacteria towards a chemical signal,
and given by
\begin{align} \label{prob:ks}
  \begin{cases}
    u_t = \Delta u - \nabla \cdot (u \nabla v), \\
    v_t = \Delta v - v + u,
  \end{cases}
\end{align}
admits solutions blowing up in finite time.

While all solutions are global and bounded in one dimension \cite{OsakiYagiFiniteDimensionalAttractor2001},
blow-up does occur in two- 
\cite{HorstmannWangBlowupChemotaxisModel2001, SenbaSuzukiParabolicSystemChemotaxis2001, HerreroVelazquezBlowupMechanismChemotaxis1997}
and higher-dimensional \cite{WinklerFinitetimeBlowupHigherdimensional2013} domains.
For a broader introduction to \eqref{prob:ks} and similar systems we refer to the surveys \cite{BellomoEtAlMathematicalTheoryKeller2015,lan_win_survey}.
 
On the other hand, adding logistic terms to \eqref{prob:ks},
for instance in order to model
population dynamics \cite{HillenPainterUserGuidePDE2009, ShigesadaEtAlSpatialSegregationInteracting1979}
or tumor invasion processes \cite{ChaplainLolasMathematicalModellingCancer2005},
leads to the system
\begin{align*}
  \begin{cases}
    u_t = \Delta u - \nabla \cdot (u \nabla v) + \lambda u - \mu u^{1+\kappa}, \\
    v_t = \Delta v - v + u,
  \end{cases}
\end{align*}
where $\lambda, \mu \gt 0$ and $\kappa = 1$ are given parameters.
Here all solutions are global and bounded in two \cite{OsakiEtAlExponentialAttractorChemotaxisgrowth2002}
and, provided that $\mu \ge \mu_0$ for some $\mu_0 \gt 0$ depending on the space dimension,
also in higher dimensions~\cite{WinklerBoundednessHigherdimensionalParabolicparabolic2010}. (See also \cite{TelloWinklerChemotaxisSystemLogistic2007} for the corresponding parabolic--elliptic system.) 
Moreover, without any restriction on $\mu \gt 0$, in all space dimensions global weak solutions have been constructed,
which in three-dimensional convex domains additionally become eventually smooth if $λ$ is small enough~%
\cite{LankeitEventualSmoothnessAsymptotics2015}.

In order to better understand the relative strengths of the possibly explosion-enhancing cross-diffusive chemotaxis effect and the damping force of the logistic terms, also systems with weaker damping (e.g.\ $κ<1$) have been investigated with respect to the existence of classical, weak or generalized solutions (see e.g.\ \cite{NakaguchiOsakiGlobalExistenceSolutions2018,giuseppeI,giuseppeII,tian,jianlu_mario}). 

For a more complete answer, however, it seems indispensable to also search for the opposite case, that of blow-up: What happens for $\kappa \in (0, 1)$ (or for $\kappa = 1$ and small $\mu \gt 0$)?
For which values of $\kappa$ can solutions blowing up in finite time be constructed?

Even beyond solutions that grow on smaller time-scales in case of slow diffusion \cite{WinklerHowFarCan2014, LankeitChemotaxisCanPrevent2015}, some partial results in this direction are available:
The first blow-up result for a chemotaxis system with superlinear degradation
apparently goes back to \cite{WinklerBlowupHigherdimensionalChemotaxis2011}. 
Following a simplification introduced in \cite{JagerLuckhausExplosionsSolutionsSystem1992} by Jäger and Luckhaus, there it was shown that for the system
\begin{align} \label{prob:jl_ls_const_intro}
  \begin{cases}
    u_t = \Delta u - \nabla \cdot (u \nabla v) + \lambda u - \mu u^{1+\kappa}, \\
    0   = \Delta v - \ol M(t) + u, \quad \ol M(t) = \frac1{|\Omega|} \intom u
  \end{cases}
\end{align}
in a ball in $\R^n$, $n \ge 5$, finite-time blow-up is possible provided that $\lambda, \mu \gt 0$
and $\kappa \lt \frac{1}2 + \frac1{2n-2}$.
Moreover, chemotactic collapse may even happen in the physically (most) relevant space dimension three.
In \cite{WinklerFinitetimeBlowupLowdimensional2018} it was shown that 
the parabolic--elliptic system
\begin{align} \label{prob:pe_ls} 
  \begin{cases}
    u_t = \Delta u - \nabla \cdot (u \nabla v) + \lambda u - \mu u^{1+\kappa}, \\
    0   = \Delta v - v + u
  \end{cases}
\end{align}
in a ball in $\R^n$, $n \ge 3$, admits solutions blowing up in finite time
provided $\lambda, \mu \gt 0$ and
\begin{align*}
  \kappa \lt \kappape(n) \defs
  \begin{cases}
    \frac{1}6,            & n \in \{3, 4\}, \\
    \frac{1}{2n-2}, & n \gt 4.
  \end{cases}
\end{align*}
This result has also recently been transferred to a setting where two species are attracted by the same signal in \cite{tu_qiu_twospecies}, where only limited adaptations were necessary to retrieve the same values for $\kappape(n)$, and to a system with a weaker sensitivity function, \cite{YuyaTomomi}.

In two-dimensional domains, however, known results seem to be limited to the case of space-dependent functions $\mu$.
That is, if one replaces the first equation in \eqref{prob:jl_ls_const_intro} with
$u_t = \Delta u - \nabla \cdot (u \nabla v) + \lambda u - \mu_1 |x|^\alpha u^{1+\kappa}$,
then solutions blowing up in finite-time have been constructed
if again the domain is a ball, $\lambda, \mu_1, \alpha \gt 0$ and $\kappa \lt \frac{\alpha}{2}$.
Phrased differently, given any $\kappa \gt 0$, there exist blow-up solutions even in 2D---%
provided $\alpha$ is large enough \cite{FuestFinitetimeBlowupTwodimensional2020}.

There is another effect that can hinder blow-up and is often included in the model, be it for reasons of biological modelling, for example of tumour cells, cf.\ \cite{SRLC09} or \cite{Kowalczyk-PreventingBlowUp-JMAA05,gambaPercolationMorphogenesisBurgers2003}, or from a purely mathematical motivation: Nonlinear, porous-medium type diffusion (i.e.\ the replacement of $Δu$ by, e.g., $∇\cdot((u+1)^{m-1}∇u)$). If sufficiently strong, it can prevent blow-up even on its own (see \cite{HorstmannWinklerBoundednessVsBlowup2005,sugiyama07, IshidaEtAlBoundednessQuasilinearKeller2014, TaoWinklerBoundednessQuasilinearParabolic2012}  for boundedness results in case of $m>2-\f2n$) or at least in combination with logistic sources, \cite{xie_xiang,zheng_JDE,zheng_ZAMM,zheng_JMAA,wang_li_mu}.

However, for the regime of slightly weaker diffusion, the occurrence of blow-up may still be possible. And indeed, in the absence of logistics ($κ=0$, $μ=λ$), $m<\f{2n-2}n$ for $\Omega$ being a ball in $ℝ^n$ allows for some unbounded solutions (i.e.\ blow-up after either finite or infinite time, see \cite{HorstmannWinklerBoundednessVsBlowup2005, WinklerDoesVolumefillingEffect2009}), with finite-time blow-up having been detected in \cite{CieslakStinnerFinitetimeBlowupGlobalintime2012}.
As to blow-up for different combinations of diffusion and sensitivity terms, refer to \cite{cieslak_stinner_14,cieslak_stinner_15,win_itbu} or to \cite{ishida_yokota,ishida_ono_yokota} for the case of degenerate diffusion.

Parabolic--elliptic analogues were investigated in \cite{win_djie,CieslakWinklerFinitetimeBlowupQuasilinear2008,lan_itbu}, revealing blow-up after finite or infinite time for different parameter ranges. 

If logistics and nonlinear diffusion both are incorporated into the model, at least in space dimensions $n\ge 5$ sometimes blow-up is possible: In \cite{LinEtAlBlowupResultQuasilinear2018} it was shown that the techniques of \cite{WinklerBlowupHigherdimensionalChemotaxis2011} can be applied for diffusion rates $m\in[1,\frac{2n-4}{n})$ and dampening exponents $\kappa\in(0,\frac{mn}{2(n-1)})$, resulting in finite-time blow-up for some radial solutions of the system 
\begin{align} \label{prob:jl_ls_space} \tag{JL}
  \begin{cases}
    u_t = \nabla \cdot ((u+1)^{m-1} \nabla u - (u \nabla v)) + \lambda(|x|) u - \mu(|x|) u^{1+κ}, & \text{in $\Omega \times (0, \tmax)$} \\
    0   = \Delta v - \ol M(t) + u, \quad \ol M(t) \defs \frac{1}{|\Omega|} \intom u, & \text{in $\Omega \times (0, \tmax)$} \\
    \partial_\nu u = \partial_\nu v = 0,                                             & \text{on $\partial \Omega \times (0, \tmax)$}\\
    u(\cdot, 0) = u_0,                                                               & \text{in $\Omega$}, \\
  \end{cases}
\end{align}
in which the signal's equation is simplified analogously to the system from the famous work \cite{JagerLuckhausExplosionsSolutionsSystem1992} by Jäger and Luckhaus. In \cite[Remark 1.2]{LinEtAlBlowupResultQuasilinear2018}, it was conjectured that blow-up should occur for $m<\frac{2n-2}{n}$ but the answer was left open to further research. For an earlier extension of the same methods from \cite{WinklerBlowupHigherdimensionalChemotaxis2011} to systems with nonlinear diffusion and logistics, combined with superlinear sensitivity functions, see \cite{zheng_mu_hu}. 

The recent advances of \cite{WinklerFinitetimeBlowupLowdimensional2018}  in the linear-diffusion case with logistic raise some hope that also in a nonlinear setting, the discovery of blow-up is also possible in the slightly less simplified parabolic--elliptic system 
\begin{align} \label{prob:pe_ls_space} \tag{PE}
  \begin{cases}
    u_t = \nabla \cdot ((u+1)^{m-1} \nabla u - (u \nabla v)) + \lambda(|x|) u - \mu(|x|) u^{1+\kappa}, & \text{in $\Omega \times (0, \tmax)$} \\
    0   = \Delta v - v + u,                                                          & \text{in $\Omega \times (0, \tmax)$} \\
    \partial_\nu u = \partial_\nu v = 0,                                             & \text{on $\partial \Omega \times (0, \tmax)$}\\
    u(\cdot, 0) = u_0,                                                               & \text{in $\Omega$} \\
  \end{cases}
\end{align}
and, more importantly, even in the physically more relevant case of $n=3$, for instance.

\paragraph{Main results}
The present article is dedicated to this question. Aiming for blow-up, we study \eqref{prob:jl_ls_space} and \eqref{prob:pe_ls_space} in a ball $\Omega \subset \R^n$, $n \ge 3$, for sufficiently smooth nonnegative functions $\lambda, \mu$ and a parameter $κ\ge0$.
We refer to the introduction of~\cite{FuestFinitetimeBlowupTwodimensional2020} for a motivation for logistic source terms with spatial dependence.

We extend the methods of \cite{WinklerFinitetimeBlowupLowdimensional2018} to nonlinear diffusion, and show that they are applicable in \eqref{prob:jl_ls_space} as well as in \eqref{prob:pe_ls_space}. 
At the same time, we would like to offer a different perspective on these, seeing them as a bridge connecting \textit{pointwise upper  estimates} of solutions to the occurrence of explosions. 
We therefore give our main result in the following form:

\begin{theorem}\label{th:main}
  Let $n \ge 3$, $\Omega \defs B_R(0)$, $R \gt 0$, $M_0 \gt 0$, $M_1\in(0,M_0)$, $\alpha \ge 0$, $\mu_1 \gt 0$,
  $p \ge n$, $T \gt 0$ as well as $K \gt 0$
  and suppose that $\lambda, \mu$ are such that
  \begin{align}\label{eq:main:cond_lambda_mu}
    0 \le \lambda \in C^{1+\beta}([0, R])
    \quad \text{and} \quad
    0 \le \mu \in C^\beta([0, R]) \cap C^{1+\beta}((0, R])
    \qquad \text{for some $\beta \in (0, 1)$}
  \end{align}
  and comply with
  \begin{align}\label{eq:main:cond_mu}
    \mu(r) \le \mu_1 r^\alpha \qquad \text{for all $r \in [0, R]$}.
  \end{align}
  Assume moreover that $\kappa \ge 0$ and $m \gt0$ satisfy
   \begin{align*}
   m<1+\frac{n-2}{p}
   \end{align*}
   and   
  \begin{alignat}{2}
    0 \le \kappa &\lt \frac{\alpha}{p} +  \min\left\{\frac{n}{2p}, \frac{n - 2}{p} - (m-1)_+ \right\}&&\qquad\text{if } m\ge \f2p\; \label{eq:main:cond_kappa_m:1}\\
    \text{ or }\quad 0\le κ&<\f{α}p+\min\left\{\f{n}{2p},\f{n-1}p-\f{m}2\right\} &&\qquad \text{if } m\in\left(0,\f2p\right). \label{eq:main:cond_kappa_m:2}
  \end{alignat}

  Then we can find $r_1 \in (0, R)$ with the following property:
  If $$(u, v)\in\left(C^0(\ombar\times[0,\tmax))\times C^{2,1}(\ombar\times(0,\tmax))\right)^2$$ is a classical solution to \eqref{prob:jl_ls_space} or \eqref{prob:pe_ls_space}
  in $\Ombar \times [0, \tmax)$ for some $\tmax \in (0, \infty]$
  with
  \begin{align}\label{eq:main:cond_u0}
    u_0 \in \bigcup_{\beta \in (0, 1)} \con\beta \quad \text{being nonnegative, radially symmetric and radially decreasing}
  \end{align}
  and
  \begin{align}\label{eq:main:cond_mass}
    \intom u_0 = M_0
    \quad \text{but} \quad
    \int_{B_{r_1}(0)} u_0 \ge M_1
  \end{align}
  as well as
  \begin{align}\label{condition:upperestimate}
    \sup_{t \in (0, \min\{T, \tmax\})} u(x, t) \le K |x|^{-p}
    \qquad \text{for all $x \in \Omega$},
  \end{align}
  then $(u, v)$ blows up in finite time in the sense
  that $\tmax \lt \infty$ and
  \begin{align} \label{eq:main:blow_up}
    \limsup_{t \nea \tmax} \|u(\cdot, t)\|_{\leb \infty} = \infty.
  \end{align}
\end{theorem}

If upper estimates as in \eqref{condition:upperestimate} are known, this theorem shows that finite-time blow-up is possible in \eqref{prob:pe_ls_space} and \eqref{prob:jl_ls_space}, that is, there are initial data such that \eqref{eq:main:blow_up} holds with some $\Tmax<\infty$. This results in the following:

\begin{theorem}\label{th:ftbu_jl_pe}
  Let $n \ge 3$, $\Omega \defs B_R(0)$, $R \gt 0$, $M_0 \gt 0$, $\alpha \ge 0$, $\mu_1 \gt 0$,
  $p \ge n$ and
  \begin{align*}
    m \lt \frac{2n-2}{n}.
  \end{align*}
  Assume moreover that $\lambda, \mu$ satisfy \eqref{eq:main:cond_lambda_mu} and \eqref{eq:main:cond_mu}.

  \textbf{(i)}
    Suppose additionally $-\lambda', \mu' \ge 0$ in $(0, R)$.
    If
    \begin{alignat}{2}
      0 \le \kappa &\lt \frac{\alpha}{n} +  \min\left\{\frac{1}{2}, \frac{n - 2}{n} - (m-1)_+ \right\}&&\qquad\text{if } m\ge \f2n \label{eq:ftbu_jl_pe:cond_kappa_1} \\
      \text{ or }\quad 0 \le κ&<\f{α}n+\min\left\{\f{1}{2},\f{n-1}n-\f{m}2\right\} &&\qquad \text{if } m\in\left(0,\f2n\right), \label{eq:ftbu_jl_pe:cond_kappa_2}
    \end{alignat}
    then there exists $0 \le u_0 \in \con0$ with $\intom u_0 = M_0$ leading to finite-time blow-up,
    that is, there is a classical solution $(u, v)$ to \eqref{prob:jl_ls_space} fulfilling \eqref{eq:main:blow_up} for some finite $\tmax$.
    Moreover, for $m = 1$, the condition \eqref{eq:ftbu_jl_pe:cond_kappa_1}--\eqref{eq:ftbu_jl_pe:cond_kappa_2} is equivalent to
    \begin{align*}
      0 \leq \kappa \lt
      \begin{cases}
        \frac{1}3 + \frac{\alpha}{3}, & n = 3, \\[0.5em]
        \frac{1}2 + \frac{\alpha}{n}, & n \ge 4.
      \end{cases}
    \end{align*}
  \textbf{(ii)}
    Provided that $m \ge 1$ and
    \begin{align}\label{eq:ftbu_jl_pe:cond_kappa_3}
      0 \leq \kappa \lt \frac{\alpha [(m-1)n + 1]}{n(n-1)} + \min\left\{\frac{(m-1) n + 1}{2(n-1)}, \frac{n - 2 - (m-1) n}{n(n-1)} \right\},
    \end{align}
    an initial datum $0 \le u_0 \in \con0$ with $\intom u_0 = M_0$
    and a solution $(u, v)$ of \eqref{prob:pe_ls_space} satisfying \eqref{eq:main:blow_up} for some $\tmax \lt \infty$ can be constructed.
    Furthermore, for $m = 1$, the condition \eqref{eq:ftbu_jl_pe:cond_kappa_3} reads
    \begin{align*}
      0 \leq \kappa \lt
      \begin{cases}
        \frac{1}6 + \frac{\alpha}{6},            & n = 3, \\[0.5em]
        \frac1{2(n-1)} + \frac{\alpha}{n(n-1)}, & n \ge 4.
      \end{cases}
    \end{align*}
\end{theorem}

Before we provide a more detailed comparison to the conditions on the existence of solutions blowing up in finite-time established in previous works, let us note the following.
\begin{remark}
\begin{enumerate}
\item The condition $m \lt \frac{2n-2}{n}$ in Theorem~\ref{th:ftbu_jl_pe} is optimal:
Even without any degradation terms (i.e.\ $\lambda = \mu \equiv 0$), classical solutions to \eqref{prob:jl_ls_space} and \eqref{prob:pe_ls_space} are always global in time for $m \gt \frac{2n-2}{n}$ (cf.\ \cite{win_djie} and \cite{lan_itbu}, respectively).
In particular, 
for \eqref{prob:jl_ls_space} the upper bound on the admissible diffusion exponents in Theorem~\ref{th:ftbu_jl_pe}~(i) coincides with the conjectured critical exponent in \cite[Remark 1.2]{LinEtAlBlowupResultQuasilinear2018}.

\item The assumption $p \ge n$ in Theorem~\ref{th:main} is a natural limitation.
  In fact, since \eqref{condition:upperestimate} implies $\intom u^q \le K^q \omega_{n-1} \int_0^R r^{n-1 + pq} \dr \lt C(q)$
  for all $q \lt \frac np$ and certain $C(q) \gt 0$,
  assuming that \eqref{condition:upperestimate} hold for some $p \lt n$ and a large class of initial data,
  these initial data would automatically be uniformly bounded in $\leb{\frac{n+p}{2p}}$, say, by $C'$.
However, as can be seen by applying Hölder's inequality, their mass on $B_{r_1}(0)$ would then be bounded by
$C' |B_{r_1}(0)|^\frac{n-p}{n+p}$, which converges to $0$ for $r_1 \sea 0$. Thus, it would not be clear if one of these initial data could still fulfill \eqref{eq:main:cond_mass} for the value of $r_1$ given by Theorem~\ref{th:main}.

\item To the best of our knowledge, Theorem~\ref{th:ftbu_jl_pe} provides the first detection of finite-time blow-up for Keller--Segel systems with nonlinear diffusion and superlinear damping terms in space dimensions $3$ and $4$. For \eqref{prob:pe_ls_space} and $m\neq1$ it is furthermore the first such result in higher dimensions.
\item The finite-time blow-up result for \eqref{prob:pe_ls_space} also constitutes a partial answer to the second part of Open Problem (i) in \cite{openproblem}.
\end{enumerate}
\end{remark}

Now, let us take a more in-depth look at the new ranges for the parameter $\kappa$ in some different spatial dimensions under the assumption of $\alpha=0$ for some special values of $m$. In this setting, earlier works have established a certain $\kappa^*$ (provided in Table~\ref{table}) for which blow-up has been proven for $\kappa<\kappa^*$.  
\renewcommand{\arraystretch}{1.5}
\setlength{\tabcolsep}{10pt}
\begin{table}\label{table}\centering 
\begin{tabular}{c|ccccc}\toprule
Work &\cite{WinklerBlowupHigherdimensionalChemotaxis2011}&\cite{LinEtAlBlowupResultQuasilinear2018} & \cite{WinklerFinitetimeBlowupLowdimensional2018} &\multicolumn{2}{c}{present article}\\
System &\eqref{prob:jl_ls_space}&\eqref{prob:jl_ls_space}&\eqref{prob:pe_ls_space}&\eqref{prob:jl_ls_space}&\eqref{prob:pe_ls_space}\\\midrule
$n=3$, $m=1$&& &$\frac16$& $\frac{1}{3}$ &$\frac16$\\
$n=4$, $m=1$&& &$\frac16$& $\frac{1}{2}$ &$\frac16$\quad\\
$n\geq5$, $m=1$&$\frac{n}{2(n-1)}$&$\frac{n}{2(n-1)}$ & $\frac{1}{2(n-1)}$ &$\quad\frac{1}{2}\quad$ &$\frac1{2(n-1)}$\\
$n\geq5$, $m\in\left(1,\frac{2n-4}{n}\right)$&&$\frac{nm}{2(n-1)}$&&$\frac{1}{2}$&$\star$\\
$n\geq3$, $m\in\left(1,\frac{2n-2}{n}\right)$&&&&$\star$&$\star$\\\bottomrule
\end{tabular} 
\caption{Supremum of the range of $\kappa$ for which finite-time blow-up has been detected. Here, $\star$ means that for the prescribed values of $n$ and $m$, Theorem~\ref{th:ftbu_jl_pe} asserts the existence of solutions blowing up in finite time for certain $\kappa \gt 0$, for whose precise values we refer to Theorem~\ref{th:ftbu_jl_pe}.
}
\end{table}

Evidently, the findings of \cite{WinklerBlowupHigherdimensionalChemotaxis2011} and \cite{LinEtAlBlowupResultQuasilinear2018} (and also of the related \cite{zheng_mu_hu}) only cover higher dimensions. For $n\geq5$ and $m\in[1,\frac{2n-4}{n})$, however, these results still provide better ranges than the one we could attain with our method. To the best of our knowledge, for larger values of $m$ or for small space dimensions, however, our results provides the first proof of finite-time blow-up in \eqref{prob:jl_ls_space}. 

Regarding \eqref{prob:pe_ls_space}, for the linear diffusion case we are able to match the range previously established in \cite{WinklerFinitetimeBlowupLowdimensional2018}, while also providing first results for the nonlinear diffusion setting in higher dimensions.

When comparing the parameter ranges across the two different systems for $m = 1$ and $n \in \{2, 3\}$,
we see that our results for \eqref{prob:jl_ls_space} yield a wider regime for $\kappa$ than the corresponding results obtained in \cite{WinklerFinitetimeBlowupLowdimensional2018} for \eqref{prob:pe_ls_space}. Indeed, $\frac13 \gt \frac16$ and $\frac12 \gt \frac16$.
In general, known results for \eqref{prob:jl_ls_space} are stronger than for \eqref{prob:pe_ls_space}.
However, lacking  global existence results for $\kappa \lt 2$,
it is yet unclear whether blow-up is actually more prominent in \eqref{prob:jl_ls_space} or just easier to detect.

   

\paragraph{Main ideas}
As is meanwhile well-established in the context of finite-time blow-up proofs for chemotaxis systems
and has first been proposed by Jäger and Luckhaus in \cite{JagerLuckhausExplosionsSolutionsSystem1992},
we consider the mass accumulation function
\begin{align*}
  w(s, t) \defs \int_0^{s^\frac1n} \rho^{n-1} u(\rho, t) \drho, \quad s \in [0, R^n], t \in [0, \tmax),
\end{align*}
which transforms \eqref{prob:jl_ls_space} into the scalar equation
\begin{align*}
        w_t  
  &=    n^2 s^{2-\frac2n} w_{ss}
          + n w w_s
          - n \ol m(t) s w_s
          + n \int_0^s \lambda(\sigma^\frac1n) w_s(\sigma, t) \dsigma
          - n \int_0^s \mu(\sigma^\frac1n) w_s^{κ+1}(\sigma, t) \dsigma
\end{align*}
(and \eqref{prob:pe_ls_space} at least into a system that is easier to handle than \eqref{prob:pe_ls_space} itself).
The main difficulty for detecting finite-time blow-up lies in the fact that
the term $+n w w_s$, stemming from the cross diffusion in \eqref{prob:jl_ls_space}, has to counter the quite different terms
$n^2 s^{2-\frac2n} w_{ss}$ and $- n \int_0^s \mu(\sigma^\frac1n) w_s^2(\sigma, t) \dsigma$
originating from the diffusion and logistic terms, respectively.

Following \cite{WinklerFinitetimeBlowupLowdimensional2018},
our approach consists of showing that for certain initial data, $\gamma \in (0, 1)$ and $s_0 \in (0, R^n)$, the function
\begin{align*}
    \phi(s_0,\cdot): [0, \tmax) \ra \R, \quad
    t \mapsto \intns s^{-\gamma} (s_0-s) w(s, t) \ds
\end{align*}
cannot exist globally in time, which due to the blow-up criterion asserted in Lemma~\ref{lm:local_ex}
implies the desired finite-time blow-up result \eqref{eq:main:blow_up}.
That is, in Section~\ref{sec:phi} we show that $\phi$ is a supersolution to the ODI $\phi' = a \phi^2 - b$ for certain $a,b \gt 0$
and in Section~\ref{sec:proof_th11} we conclude the existence of initial data leading to finite-time blow-up of $\phi$ and hence $u$.

Let us briefly discuss how we deal with the two most problematic terms stemming from the degradation and diffusion terms, respectively.
As we will see in Lemma~\ref{lm:i4}, in order to handle the former, we essentially need to control
\begin{align}\label{eq:intro:i4}
  -\int_0^{s_0} s^\frac{\alpha}{n} (s_0-s) w_s^{1+\kappa}(s, t) \ds.
\end{align}
At this point, the assumption \eqref{condition:upperestimate} comes into play,
which due to $w_s(s, t) = \frac{u(s^\frac1n, t)}{n}$ for $(s, t) \in [0, R^n] \times [0, \tmax)$
implies $w(s, t) \le \frac{C}{n} s^\frac pn$ for $(s, t) \in [0, R^n] \times [0, \tmax)$.
Thus, as a starting point, we can apply this estimate to $w_s^\kappa$ in \eqref{eq:intro:i4} and then integrate by parts.
Moreover, by \eqref{eq:i1:eq1}, the term arising from the diffusion can be estimated against (some positive multiple of)
\begin{align*}
  -\intns s^{1-\frac{2}{n}-\gamma}(s_0-s)(nw_s+1)^m\ds.
\end{align*}
For $m \ge 1$, we can proceed as above, that is, we apply the pointwise upper bound to $w_s^{m-1}$ and then integrate by parts,
while for $m \lt 1$ we can follow at least two different paths:
For $m \in (0, \frac2p)$, we apply this bound to $w^m$ and do not integrate by parts
and for $m \ge \frac2p$, we estimate $(n w_s + 1)^m \le n w_s + 1$ and integrate by parts without using the pointwise upper estimate for $w_s$ at all.
The fact that depending on the value of $m$ we employ two different methods here 
is the reason for the different conditions in \eqref{eq:main:cond_kappa_m:1} and \eqref{eq:main:cond_kappa_m:2}.

At last, we show that pointwise upper estimates of the form \eqref{condition:upperestimate} are indeed available
both for \eqref{prob:jl_ls_space} and \eqref{prob:pe_ls_space}.
While for the former system we make use of the comparison principle applied to $u_r$ in Lemma~\ref{lm:u_pw_bdd_jl},
for the latter we resort to the recent study on blow-up profiles \cite{FuestBlowupProfilesQuasilinear2020}
to obtain the desired bounds in Lemma~\ref{lm:u_pw_bdd_pe}.

\section{Preliminaries}\label{sec:prelim}
We henceforth always assume $n \ge 3$ and $\Omega \defs B_R(0) \subset \R^n$ for some $R \gt 0$.
Furthermore, in Sections~\ref{sec:prelim}--\ref{sec:proof_th11}, we also fix $m \gt 0$, $\kappa \ge 0$, $\alpha\ge 0$, $M_0>0$, $M_1\in(0,M_0)$, $\lambda_1 \gt 0$ as well as
functions $\lambda, \mu$ complying with $\lambda \le \lambda_1$, \eqref{eq:main:cond_lambda_mu} and \eqref{eq:main:cond_mu}.
To simplify the notation, we also fix an initial datum $u_0$ satisfying \eqref{eq:main:cond_u0} with $\intom u_0 = M_0$,
but emphasize that all constants below, unless otherwise stated, are independent of $u_0$.

By $(u,v)$ we will refer to a solution to either of the systems \eqref{prob:jl_ls_space} or \eqref{prob:pe_ls_space}, and we also set $\ol M(t) \defs |\Omega|^{-1}\io u(\cdot,t)$ for $t \in [0, \tmax)$.

\begin{lemma} \label{lm:local_ex}
  Suppose that $u_0 \colon \Ombar \ra [0, \infty)$ is   Hölder continuous. 
  Then for each of the systems \eqref{prob:jl_ls_space} and \eqref{prob:pe_ls_space} there exist $\tmax \in (0, \infty]$ and a classical solution $(u, v)$,
  uniquely determined by
  \begin{align}
    u &\in C^0(\ombar \times [0, \tmax)) \cap C^{2, 1}(\ombar \times (0, \tmax)), \label{existence:regularity-of-u}\\
    v &\in \bigcap_{q \gt n} C^0([0, \tmax); W^{1, q}(\Omega)) \cap C^{2, 1}(\ombar \times (0, \tmax)) \nonumber
  \end{align}
  and, in case of \eqref{prob:jl_ls_space}, 
  \begin{align*}
    \intom v(\cdot, t) = 0 \quad \text{for all $t \in (0, \tmax)$}.
  \end{align*}
  
  Moreover, $u \ge 0$ in $\Omega \times (0, \tmax)$ and if $\tmax \lt \infty$, then
  \begin{align*}
    \limsup_{t \nea \tmax} \|u(\cdot, t)\|_{\leb \infty} = \infty.
  \end{align*}
If, finally, $u_0$ is radially symmetric, then so are $u(\cdot,t)$ and $v(\cdot,t)$ for any $t\in(0,\Tmax)$.
\end{lemma}
\begin{proof}
  Local existence can be proved by a standard fixed point argument,
  which is explained in more detail in
  \cite{CieslakWinklerFinitetimeBlowupQuasilinear2008} or \cite{TelloWinklerChemotaxisSystemLogistic2007},
  for instance,
  while nonnegativity of $u$ follows by the maximum principle
  and preservation of radial symmetry is a consequence of uniqueness.
\end{proof}

As a first basic observation, we note that, at least locally in time, the mass of $u$ can be controlled by the parameters we fixed above---%
and thus, independently of the precise choice of $u_0$.

\begin{lemma} \label{lm:mass_ineq}
  For all $t \in (0, \tmax)$, we have
  \begin{align*}
    \intom u(\cdot, t) \le M_0 \ure^{\lambda_1 t}.
  \end{align*}
\end{lemma}
\begin{proof}
  Due to $λ\le λ_1$ and nonnegativity of $μ$, integrating the first equation over $\Omega$ gives 
  \begin{align*}
        \ddt \intom u
    =   \intom \lambda u - \intom \mu u^{1+\kappa}
    \le \lambda_1 \intom u 
    \qquad \text{in $(0, \tmax)$},
  \end{align*}
  so that the statement follows by an ODE comparison argument.
\end{proof}

\section{Proving finite-time blow-up}\label{sec:phi}
Following \cite{JagerLuckhausExplosionsSolutionsSystem1992,
biler_hilhorst_nadzieja_II,
WinklerFinitetimeBlowupLowdimensional2018}, we define
\begin{align*}
  w(s, t) \defs \int_0^{s^\frac1n} \rho^{n-1} u(\rho, t) \drho, \quad s \in [0, R^n], t \in [0, \tmax)
\end{align*}
and, given $s_0\in(0,R^n)$ and $γ\in(0,1)$, introduce the functions 
  \begin{align}\label{def:phi}
    \phi(s_0,\cdot): [0, \tmax) \ra \R, \quad
    t \mapsto \intns s^{-\gamma} (s_0-s) w(s, t) \ds
  \end{align}
(cf.\ \cite[equation~(4.1)]{WinklerFinitetimeBlowupLowdimensional2018}) and 
  \begin{align*}
    ψ(s_0,\cdot): [0, \tmax) \ra \R, \quad
    t \mapsto \intns s^{-\gamma} (s_0-s) w(s, t)w_s(s,t) \ds.
  \end{align*}
If, by the usual slight abuse of notation, we identify the radially symmetric function $u\in C^0(\Ombar\times[0,\Tmax))$ with $u\in C^0([0,R]\times[0,\Tmax))$ and write $u_r$ for its radial derivative, we can compute the spatial derivatives of $w$:
\begin{lemma}We have 
\begin{equation}\label{w:regularity}
 w\in C^{1,0}([0,R^n]\times[0,\Tmax))\cap C^{2,1}([0,R^n]\times(0,\Tmax))\cap C^{3,0}((0,R^n]\times(0,\Tmax)),
\end{equation}
and
\begin{equation}\label{ws-and-wss}
 w_s(s,t)=\f1n u(s^{\f1n},t), \quad w_{ss}(s,t)=\f1{n^2}s^{\f1n-1}u_r(s^{\f1n},t) \qquad\text{for } s\in(0,R^n], t\in(0,\Tmax)
\end{equation}
and, with $K$ and $T$ from \eqref{condition:upperestimate}, 
\begin{equation}\label{upperestimate:ws}
 w_s(s,t)\le \f{K}n s^{-\f pn}\qquad \text{for } (s,t)\in(0,R^n]\times(0,T).
\end{equation}
\end{lemma}
\begin{proof}
 For the regularity, we rely on \eqref{existence:regularity-of-u}; the final estimate \eqref{upperestimate:ws} results from \eqref{ws-and-wss} and \eqref{condition:upperestimate}.
\end{proof}
For $\phi$, which we later want to show to blow up, the following differential inequality holds:

\begin{lemma} \label{lm:phi_ode}
  For any choice of $\gamma \in (0, 1)$ and $s_0 \in (0, R^n)$ the function $\phi$ of \eqref{def:phi} 
  belongs to $C^0([0, \tmax)) \cap C^1((0, \tmax))$ and fulfills
  \begin{subequations}\label{eq:phi_ode:statement}
  \begin{align} 
          \phi'(s_0,t)
    &\ge  n^2 \intns s^{2-\frac2n-\gamma} (s_0-s) (n w_s + 1)^{m-1} w_{ss}(s, t) \ds \notag \\
    &\pe  + n \intns s^{-\gamma} (s_0-s) w(s, t) w_{s}(s, t) \ds \notag \\
    &\pe  - \ol M(t) \intns s^{1-\gamma} (s_0-s) w_s(s, t) \ds \notag \\
    &\pe  - n^{\kappa} \mu_1 \intns s^{-\gamma} (s_0-s)
           \left( \int_0^s \sigma^\frac{\alpha}{n} w_s^{1+\kappa}(\sigma, t) \dsigma \right) \ds \notag \\
    &\sfed I_1(s_0,t) + I_2(s_0,t) + I_3(s_0,t) + I_4(s_0,t)
  \end{align}
  for all $t \in (0, \tmax)$ in the case of \eqref{prob:jl_ls_space}.
  For \eqref{prob:pe_ls_space}, the same estimate holds with 
  \begin{equation}
   I_3(s_0,t)=-n\int_0^{s_0} s^{-γ}(s_0-s)w_s(s,t) z(s,t)\ds, \qquad t \in (0, \tmax),
  \end{equation}
  \end{subequations}
  where
  \[
    z(s,t) \defs \int_0^{s^{\f1n}}ρ^{n-1}v(ρ,t)\drho
    \qquad \text{for $s \in [0, R^n]$ and $t \in [0, \tmax)$.}
  \]
\end{lemma}
\begin{proof}
The regularity of $ϕ$ follows from \eqref{w:regularity}.
Written in radial coordinates, the differential equations in \eqref{prob:jl_ls_space} read
\begin{align}\label{eq:JL2radial}
\begin{cases}
 u_t(r,t) &= r^{1-n}\kl{r^{n-1}\kl{(u(r,t)+1)^{m-1}u_r(r,t)-u(r,t)v_r(r,t)}}_r+λ(r)u(r,t)-μ(r)u^{1+κ}(r,t),\\
 0 &= r^{1-n}\kl{r^{n-1}v_r(r,t)}_r -\ol M(t) + u(r,t)
\end{cases}
\end{align}
for $(r,t)\in[0,R)\times(0,\Tmax)$, 
where the second equation can be transformed into  
\begin{equation}\label{vr-radial-JL}
 v_r(r,t) = \ol M(t) r^{1-n}\int_0^r ρ^{n-1}dρ - r^{1-n}\int_0^r ρ^{n-1}u(ρ,t)dρ = \f rn \ol M(t) -r^{1-n} w(r^n,t), 
\end{equation}
for $r\in[0,R)$, $t\in (0,\Tmax)$, 
and the first equation results in
\begin{align*}
 w_t(s,t) &= \int_0^{s^{\f1n}}ρ^{n-1}u_t(ρ,t)\drho \\
 &= \int_0^{s^{\f1n}}ρ^{n-1}ρ^{1-n} \kl{r^{n-1}\kl{(u(r,t)+1)^{m-1}u_r(r,t)-u(r,t)v_r(r,t)}}_r\big\vert_{r=ρ} \drho \\
&\qquad +\int_0^{s^{\f1n}}ρ^{n-1}λ(ρ)u(ρ,t)\drho -\int_0^{s^{\f1n}}ρ^{n-1}μ(ρ)u^{1+κ}(ρ,t)\drho\\
&= s^{\f{n-1}n}u_r(s^{\f1n},t) \kl{u(s^{\f1n},t)+1}^{m-1} - s^{\f{n-1}n}u(s^{\f1n},t)v_r(s^{\f1n},t)\\&\qquad +\int_0^{s^{\f1n}}ρ^{n-1}λ(ρ)u(ρ,t)\drho -\int_0^{s^{\f1n}}ρ^{n-1}μ(ρ)u^{1+κ}(ρ,t)\drho
\end{align*}
for $(s, t) \in (0, R^n) \times (0, \tmax)$.
If we insert \eqref{ws-and-wss} into \eqref{vr-radial-JL} and use nonnegativity of $λ$, we obtain that 
  \begin{align}\label{eq:deriveIs-final}
        w_t
    \ge n^2 s^{2-\frac2n} (n w_s + 1)^{m-1} w_{ss} + n w w_s - \ol M(t) s w_s
        - n^{\kappa} \mu \int_0^s \sigma^\frac{\alpha}{n} w_s^{1+\kappa}(\sigma, t) \dsigma
  \end{align}
  in $(0, R^n) \times (0, \tmax)$,
  which by multiplication with $s^{-γ}(s_0-s)$ and integration implies the statement for \eqref{prob:jl_ls_space}.%
  
  For \eqref{prob:pe_ls_space}, in \eqref{eq:JL2radial}, $-\ol{M}(t)$ has to be replaced by $-v$, so that after essentially the same computation, in \eqref{eq:deriveIs-final} the term $-\ol{M}(t)sw_s$ is substituted by $-nw_s\int_0^{s^{\f1n}} ρ^{n-1} v(ρ,t)\drho=-nw_sz$.
\end{proof}

In the remaining part of this section, we further estimate the terms of the right hand side of \eqref{eq:phi_ode:statement},
aiming to show that $\phi(s_0, \cdot)$ fulfills a certain superlinear ODE.
These results will then be combined in Section~\ref{sec:proof_th11};
ultimately, the consolidation of the lemmata will show that at least for certain values of $\gamma$ and $s_0$,
$\phi(s_0, \cdot)$ cannot exist globally.

In order to streamline the arguments below, let us first state two elementary lemmata.

\begin{lemma} \label{lm:beta-f}
For all $\alpha>-1$ and $\beta>-1$ and any $s_0\ge 0$ we have
\begin{align*}
\intns s^\alpha(s_0-s)^\beta\ds =\Beta(α+1,β+1) s_0^{\alpha+\beta+1}.
\end{align*}
\end{lemma}
\begin{proof}
This is an evident consequence of the properties of the beta function.
\end{proof}

\begin{lemma} \label{lm:bdd_w_psi}
  Let $\gamma \in (0, 1)$ and $s_0 \in (0, R^n)$.
  Then
  \begin{align*}
        w(s, t)
    \le \sqrt2 s^\frac\gamma2 (s_0-s)^{-\frac12}
    \sqrt{ψ(s_0,t)}
  \end{align*}
  holds true for all $(s, t) \in (0, s_0) \times (0, \tmax)$.
\end{lemma}
\begin{proof}
  This inequality, in its essence based on the fundamental theorem of calculus,
  is a direct consequence of \cite[Lemma~4.2]{WinklerFinitetimeBlowupLowdimensional2018}, which for every $t\in (0,\Tmax)$ can be applied to $φ=w(\cdot,t)$, because $φ(0)=0$, $φ'\ge 0$ in $(0,s_0)$ and $φ\in C^1([0,s_0])$.
\end{proof}

The estimate \eqref{upperestimate:ws}, which originates in the crucial assumption \eqref{condition:upperestimate},
will come into play at two different places.
The first of these is the following lemma, where said upper estimate is the most important ingredient
for controlling the term arising from the logistic source, namely $I_4$ in \eqref{eq:phi_ode:statement}.

\begin{lemma} \label{lm:i4}
  Let $\gamma \in (0, 1)$ and $p \ge n$ satisfy $\frac{p \kappa}{n} -\frac{\alpha}{n}  \lt \frac{\gamma}{2}$. Whenever \eqref{condition:upperestimate}
  is fulfilled for some $K \gt 0$, $T>0$ 
  and $s_0 \in (0, R^n)$ then $I_4$ from Lemma~\ref{lm:phi_ode} satisfies%
  \begin{align*}
        I_4(s_0,t)
    \ge -K^{\kappa} C s_0^{\frac{3-\gamma}{2} - \frac{p\kappa}{n}+\frac{\alpha}{n} }\sqrt{ψ(s_0,t)}
    \qquad \text{for all $t \in (0, \min\set{T,\Tmax})$,}
  \end{align*}
 with $C=\f{μ_1\sqrt2((\f{pκ}n-\f{α}n)_++1)}{1-γ}\Beta(\f{α}n-\f{pκ}n+\f{γ}2,\f12)$.
\end{lemma}
\begin{proof}
  This can be proved analogously to \cite[Lemma~4.5]{WinklerFinitetimeBlowupLowdimensional2018}: 
  Firstly, Fubini's theorem asserts
  \begin{align*}
        I_4(s_0,t)
    &=  - c_1 \int_0^{s_0}σ^\frac{\alpha}{n} w_s^{1+\kappa}(\sigma, t) \left( \int_\sigma^{s_0} s^{-\gamma} (s_0-s) \ds \right) \dsigma \\
    &\ge  - c_1 \int_0^{s_0} σ^\frac{\alpha}{n} (s_0-σ) w_s^{1+\kappa}(\sigma, t) \left( \int_0^{s_0} s^{-\gamma} \ds \right) \dsigma \\
    &=   -c_2 s_0^{1-\gamma} \int_0^{s_0} s^\frac{\alpha}{n} (s_0-s) w_s^{1+\kappa}(s, t) \ds
  \end{align*}
  for all $t \in (0, \tmax)$,
  where $I_4$ is as in \eqref{eq:phi_ode:statement}, $c_1 \defs n^{\kappa} \mu_1 \gt 0$ and $c_2 \defs \frac{c_1}{1-\gamma} \gt 0$.
  Next, 
  we use \eqref{upperestimate:ws} and integrate by parts to see that
  \begin{align*}
    \pe  \intns s^{\frac{\alpha}{n}} (s_0-s) w_s^{1+\kappa}(s, t) \ds 
    &\le  \frac{K^{\kappa}}{n^{\kappa}} \intns  s^{\frac{\alpha}{n} - \frac{p \kappa}{n}} (s_0-s) w_s(s, t) \ds \\
    &=    \frac{K^{\kappa}}{n^{\kappa}} \left( \frac{p \kappa}{n} - \frac{\alpha}{n} \right) \intns s^{\frac{\alpha}{n} - \frac{p \kappa}{n}-1} (s_0-s) w(s, t) \ds \\
    &\pe  + \frac{K^{\kappa}}{n^{\kappa}} \intns s^{\frac{\alpha}{n} - \frac{p \kappa}{n}} w(s, t) \ds + \frac{K^{\kappa}}{n^{\kappa}} \left[ s^{\frac{\alpha}{n} - \frac{p \kappa}{n}} (s_0 - s) w(s, t) \right]_0^{s_0} \\
    &\le   K^{\kappa} c_3 s_0 \intns s^{\frac{\alpha}{n} - \frac{p \kappa}{n}  - 1} w(s, t) \ds
  \end{align*}
  holds for all $t \in (0, \min\set{T,\Tmax})$ with $c_3 \defs\frac{(\frac{p \kappa}{n} - \frac{\alpha}{n})_+ + 1}{n^{\kappa}} \gt 0$.
  Here we apply Lemma~\ref{lm:bdd_w_psi} and obtain
   \begin{align*}
           \intns s^{\frac{\alpha}{n} - \frac{p \kappa}{n}  - 1} w(s, t) \ds
     &\le  \sqrt2 \intns s^{\frac{\alpha}{n} -\frac{p \kappa}{n}  - 1 + \frac\gamma2} (s_0-s)^{-\frac12} \ds \sqrt{ψ(s_0,t)}
   \end{align*}
  for all $t \in (0, \min\set{T,\Tmax})$. Finally, we note that according to Lemma~\ref{lm:beta-f} 
  \begin{align*}
          \intns s^{\frac{\alpha}{n} - \frac{p \kappa}{n} - 1 + \frac\gamma2} (s_0-s)^{-\frac12} \ds
    =   \Beta\kl{\f{α}n-\f{pκ}n +\f{γ}2,\f12} s_0^{\frac{\alpha}{n} - \frac{p \kappa}{n} + \frac\gamma2-\frac12},
  \end{align*}
  since $\frac{\alpha}{n} - \frac{p \kappa}{n} + \frac\gamma2 \gt 0$ by assumption. The statement follows by combining the estimates above.
\end{proof}

We now turn our attention to the integral involving the effects of nonlinear diffusion.
This is the second place where (at least for certain $m$) we make use of the assumption \eqref{condition:upperestimate}.

\begin{lemma}\label{lm:i1}
 Suppose that \eqref{condition:upperestimate} holds for some $p \ge n$, $K \gt 0$ and $T>0$  
  and let $I_1$ be as in \eqref{eq:phi_ode:statement}.\\
\textbf{(i)} 
Assume that 
  \begin{align} \label{eq:i1:cond_gamma-large-m}
   0<m<1+\frac{n-2}{p}\quad&\text{and}\quad 1 - \frac2n - \frac{p}{n} (m-1)_+<\gamma < 2 - \frac4n - \frac{2p}{n} (m-1)_+.
  \end{align}
  Then there is $C \gt 0$ such that for any $s_0 \in (0, R^n)$ and all $t\in(0, \min\set{T,\Tmax})$ we have
  \begin{align*}
        I_1(s_0,t)\ge -
    C s_0^{\frac{3-\gamma}{2} - \frac2n - \frac{p}{n}(m-1)_+} 
          \sqrt{ψ(s_0,t)} -C s_0^{3-\gamma-\frac2n}.
  \end{align*}
\textbf{(ii)} Assume that 
\begin{align}
  0<m<\min\left\{1,\frac{2(n-1)}{p}\right\}\quad&\text{and}\quad 0<\gamma<2-\frac{2}{n}-\frac{pm}{n}.\label{eq:i1:cond_gamma-small-m}
  \end{align}

  Then there is $C \gt 0$ such that for any $s_0 \in (0, R^n)$ and all $t\in(0, \min\set{T,\Tmax})$ we have
  \begin{align*}
        I_1(s_0,t)\ge 
          -C s_0^{3-\gamma-\frac{2}{n}-\frac{p}{n}m}-C s_0^{3-\gamma-\frac{2}{n}}.
  \end{align*}

\end{lemma}

\begin{proof}
  Direct calculation gives for every $s_0 \in (0, R^n)$
  \begin{align*}
    I_1(s_0, t)&=n^2   \intns s^{2 - \frac2n - \gamma} (s_0 - s) (n w_s + 1)^{m-1} w_{ss} \ds \nonumber\\
    &=    \frac{n}{m} \intns s^{2 - \frac2n - \gamma} (s_0 - s) ((n w_s + 1)^m)_s \ds \nonumber\\
    &=    - \frac{n}{m} \left(2 - \frac2n - \gamma \right) \intns s^{1 - \frac2n - \gamma} (s_0 - s) (n w_s + 1)^m \ds \nonumber\\
    &\pe  + \frac{n}{m} \intns s^{2 - \frac2n - \gamma} (n w_s + 1)^m \ds      + \frac{n}{m}\left[s^{2 - \frac2n - \gamma} (s_0 - s) (n w_s + 1)^m \right]_0^{s_0}
  \end{align*}
  in $(0, \tmax)$. The last two terms therein are positive, since \eqref{eq:i1:cond_gamma-large-m} and \eqref{eq:i1:cond_gamma-small-m} both entail $2 - \frac2n - \gamma \gt 0$, leading to
  \begin{align}\label{eq:i1:eq1}
  I_1(s_0, t)\geq -\frac{n}{m}\left(2-\frac{2}{n}-\gamma\right)\intns s^{1-\frac{2}{n}-\gamma}(s_0-s)(nw_s+1)^m\ds \qquad\text{for }(s_0,t)\in (0,R^n)\times(0,\tmax).
  \end{align}
Now, let us start by considering the case that \eqref{eq:i1:cond_gamma-large-m} holds.
First we find that $w_s\geq0$ implies for $m\ge1$ that $(nw_s+1)^m\leq 2^{m-1}(n^m w_s^m+1)$ in $(0,\tmax)$.
By \eqref{upperestimate:ws}
this entails that
 $(nw_s+1)^m\leq 2^{m-1} K^{m-1} n s^{-\frac{p}{n}(m-1)}w_s+2^{m-1}$ on $(0,R^n)\times(0,\min\set{T,\Tmax})$.
On the other hand, for $m\in(0,1)$ we have $(nw_s+1)^m\leq n w_s+1$ in $(0,R^n)\times(0,\min\set{T,\Tmax})$. Thus, letting $c_1\defs\max\{n,2^{m-1},2^{m-1} n K^{(m-1)_+}\}$ we find that from combining these two estimates we have 
$$(n w_s+1)^m\leq c_1 s^{-\frac{p}{n}(m-1)_+ } w_s+c_1\qquad\text{in }(0,R^n)\times(0,\min\set{T,\Tmax})$$ and hence, from \eqref{eq:i1:eq1}, 
\begin{align*}
    I_1(s_0, t)&\ge -\frac{n}{m}\left(2-\frac{2}{n}-\gamma\right)\intns s^{1-\frac{2}{n}-\gamma}(s_0-s)(nw_s+1)^m\ds\\
    &\ge  - c_1 \left(2 - \frac2n - \gamma \right)
            \intns s^{1 - \frac2n - \frac{p}{n}(m-1)_+ - \gamma} (s_0 - s) w_s \ds-c_1\left(2-\frac{2}{n}-\gamma\right)\intns s^{1-\frac{2}{n}-\gamma}(s_0-s)\ds
\intertext{
in $(0,\min\set{T,\Tmax})$ and for every $s_0\in(0,R^n)$. An integration by parts therefore yields
}
I_1(s_0, t)&\ge    c_1 \left(2 - \frac2n - \gamma \right) \left(1 - \frac2n - \frac{p}{n}(m-1)_+ - \gamma \right)
            \intns s^{- \frac2n - \frac{p}{n}(m-1)_+ - \gamma} (s_0 - s) w \ds\\
    &\pe  - c_1 \left(2 - \frac2n - \gamma \right)
            \intns s^{1 - \frac2n - \frac{p}{n}(m-1)_+ - \gamma} w \ds \\
    &\pe  - c_1 \left(2 - \frac2n - \gamma \right)
            \left[ s^{1 - \frac2n - \frac{p}{n}(m-1)_+ - \gamma}(s_0-s) w \ds \right]_0^{s_0}\\
     &\pe -c_1\left(2-\frac{2}{n}-\gamma\right)\intns s^{1-\frac{2}{n}-\gamma}(s_0-s)\ds\quad\text{in }(0,\min\set{T,\Tmax})\text{ and for every } s_0\in(0,R^n).
  \end{align*}
  The third term on the right hand side is nonnegative,
  and in the other terms we use $s_0 - s \le s_0$ and $s \le s_0$ for all $s \in (0, s_0)$ as well as the conditions $\gamma \lt 2 - \frac2n$ and $\gamma \gt 1 - \frac2n - \frac{p}{n} (m-1)_+$ contained in \eqref{eq:i1:cond_gamma-large-m}  
  to see that
  \begin{align}\label{eq:i1:ineq1}
    I_1(s_0, t)
    &\ge  - c_1 \left(2 - \frac2n - \gamma \right) \left(\gamma + \frac2n + \frac{p}{n} (m-1)_+ \right)
            s_0 \intns s^{- \frac2n - \frac{p}{n}(m-1)_+ - \gamma} w \ds-c_1 s_0^{3-\gamma-\frac{2}{n}}
  \end{align}
  in $(0, \min\set{T,\Tmax})$ and for $s_0\in(0,R^n)$. To estimate further, we make use of Lemma~\ref{lm:bdd_w_psi}, the fact that \eqref{eq:i1:cond_gamma-large-m} entails $0<1-\frac{2}{n}-\frac{p}{n}(m-1)_+-\frac{\gamma}{2}$ and Lemma~\ref{lm:beta-f} to obtain
  \begin{align}\label{eq:i1:ineq2}
    \pe  s_0 \intns s^{- \frac2n - \frac{p}{n}(m-1)_+ - \gamma} w \ds \nonumber
    &\le  \sqrt2 s_0
            \left( \intns s^{- \frac2n - \frac{p}{n}(m-1)_+ - \frac{\gamma}{2}} (s_0 - s)^{-\frac12} \ds \right) \sqrt{ψ(s_0,t)}\nn\\
    &=c_3 
            s_0^{\frac{3-\gamma}{2} - \frac2n - \frac{p}{n}(m-1)_+}\sqrt{ψ(s_0,t)}
    \qquad \text{in $(0, \tmax)$}
  \end{align}
  for $c_3=\sqrt{2}\Beta(1-\f2n-\f pn(m-1)_+-\f{γ}2,\f12) \gt 0$. Collecting \eqref{eq:i1:ineq1} and \eqref{eq:i1:ineq2} proves the estimate of $I_1$ in the case that \eqref{eq:i1:cond_gamma-large-m} holds.
    
To verify the asserted inequality in the case of \eqref{eq:i1:cond_gamma-small-m}, we return to \eqref{eq:i1:eq1} and note that due to $w_s\geq0$ and $m\in(0,1)$ we have $(nw_s+1)^m\leq n^m w_s^m+1$ on $(0,\min\set{T,\Tmax})$. Here, we rely on \eqref{upperestimate:ws} to conclude that $(nw_s+1)^m\leq K^ms^{-\frac{p m}{n}}+1$ in $(0,\min\set{T,\Tmax})$ and hence
\begin{align*}
I_1(s_0, t)&\geq -\frac{K^m n}{m}\left(2-\frac{2}{n}-\gamma\right)\intns s^{1-\frac{2}{n}-\frac{p m}{n}-\gamma}(s_0-s)\ds-\frac{n}{m}\left(2-\frac{2}{n}-\gamma\right)\intns s^{1-\frac{2}{n}-\gamma}(s_0-s)\ds
\end{align*}
in $(0,\min\set{T,\Tmax})$ and for $s_0\in(0,R^n)$, where the conditions $\gamma<2-\frac{2}{n}-\frac{p m}{n}$ and $\gamma<2-\frac{2}{n}$ contained in \eqref{eq:i1:cond_gamma-small-m} together with $s_0-s\le s_0$ entail
\begin{align*}
I_1(s_0, t)\geq -c_4 s_0^{3-\gamma-\frac{2}{n}-\frac{p m}{n}}-\f nm s_0^{3-\gamma-\frac{2}{n}}\quad\text{in }(0,\min\set{T,\Tmax})
\end{align*}
for $c_4=\f{2-\f2n-γ}{2-\f2n-\f{pm}n-γ}\cdot\f{K^mn}m >0$, completing the proof.
\end{proof}

The arguments for estimating the remaining integrals in \eqref{eq:phi_ode:statement}
rely on the following relation between $ϕ$ and $ψ$, which was also obtained in \cite[Lemma~3.4]{WinklerFinitetimeBlowupLowdimensional2018}.
\begin{lemma}\label{lm:phipsi}
Let $γ\in(0,1)$. For every $s_0\in (0,R^n)$ and $t\in(0,\Tmax)$, 
 \[
  ϕ(s_0,t)\le Cs_0^{\f{3-γ}2} \sqrt{ψ(s_0,t)},
 \]
 where $C=\sqrt2 \Beta(1-\f{γ}2,\f12)$.
\end{lemma}
\begin{proof}
By Lemma~\ref{lm:bdd_w_psi}, 
 \[
  ϕ(s_0,t)\le s_0\int_0^{s_0} s^{-γ}w(s,t)\ds\le s_0\sqrt{2}\sqrt{ψ}\int_0^{s_0}s^{-\f{γ}2}(s_0-s)^{-\f12}\ds,
 \]
and the claim follows from Lemma~\ref{lm:beta-f}.
\end{proof}

\begin{lemma}\label{lm:i2}
  Let $\gamma \in (0, 1)$.
  Then for all $s_0 \in (0, R^n)$ and all $t \in (0, \tmax)$
  \begin{align*}
    I_2(s_0,t) \ge C s_0^{\gamma-3} \phi^2(s_0,t)
  \end{align*}
  with $\phi$ and $I_2$ as in \eqref{eq:phi_ode:statement} and $C=\f{n}{2\Beta^2(1-\f{γ}2,\f12)}$.
\end{lemma}
\begin{proof}
As $I_2=nψ$, this is a corollary of Lemma~\ref{lm:phipsi}. 
\end{proof}

\begin{lemma}\label{lm:i3}
  Let $\gamma \in (0, 1)$ and $T \gt 0$.
  Then there is $C \gt 0$ such that for all $s_0 \in (0, R^n)$ and all $t \in (0, \min\{T, \tmax\})$
  \begin{align}\label{eq:i3}
  I_3(s_0,t)\ge -Cs_0^{\f2n+1-γ} - C s_0^{\frac2n}ψ(s_0,t) 
  \end{align}
  with $I_3$ as in \eqref{eq:phi_ode:statement}.
\end{lemma}
\begin{remark}\label{rem:JLbetter}
 For \eqref{prob:jl_ls_space}, we can even obtain the stronger estimate 
  \begin{align}\label{eq:rem:i3}
    I_3(s_0,t) \ge - C \ure^{\lambda_1 t} s_0^{\frac{3-\gamma}{2}} \sqrt{ψ(s_0,t)}\ge 
    -Cs_0^{\f2n+3-γ}\ure^{2λ_1t} - C s_0^{\f2n}ψ(s_0,t)
  \end{align}
  in place of \eqref{eq:i3}.
      \end{remark}
\begin{proof}[Proof of Lemma~\ref{lm:i3} and Remark~\ref{rem:JLbetter}.]
As their expressions for $I_3$ differ, we treat the cases of \eqref{prob:jl_ls_space} and \eqref{prob:pe_ls_space} separately, beginning with \eqref{prob:jl_ls_space}: 
  An integration by parts and Lemma~\ref{lm:bdd_w_psi} show
  \begin{align*}
          I_3(s_0,t)
    &=    \ol M(t) (1-\gamma) \intns s^{-\gamma} (s_0-s) w(s, t) \ds
          - \ol M(t) \intns s^{1-\gamma} w(s, t) \ds
          - 0 \\
    &\ge  - \ol M(t)\sqrt{2} \intns s^{1-\frac{\gamma}{2}} (s_0-s)^{-\frac12} \ds
          \cdot \sqrt{ψ(s_0,t)}
  \end{align*}
  for $t \in (0, \Tmax)$ and $s_0\in(0,R^n)$, so that with some $c_1>0$, $c_2>0$, 
  \[
   I_3\ge -\frac{\sqrt{2}M_0}{|\Omega|} \ure^{λ_1t}c_1 s_0^{\f{3-γ}2}\sqrt{ψ(s_0,t)}\ge -c_2 \ure^{2λ_1t}s_0^{3-γ+\f2n}-c_2 s_0^{\f2n}ψ(s_0,t)
  \]
 for $t\in(0,\Tmax)$, $s_0\in(0,R^n)$ by Lemma~\ref{lm:mass_ineq}, Lemma~\ref{lm:beta-f} and Young's inequality. This proves \eqref{eq:rem:i3} and, since $s_0^{3-γ+\f2n}\le R^{2n} s_0^{1-γ+\f2n}$ for $s_0\in(0,R^n)$, also \eqref{eq:i3}.
  
 As to \eqref{prob:pe_ls_space},
 we follow \cite[Lemma~4.8]{WinklerFinitetimeBlowupLowdimensional2018}
 but need to make some modifications to remove the additional condition $\gamma<2-\frac4n$ required there.
 To that end, we first fix $\tilde \gamma \in (\max\{\gamma - \frac4n, 0\}, \min\{\gamma, 2-\frac4n\})$
 and apply a variant of \cite[Lemma~4.7]{WinklerFinitetimeBlowupLowdimensional2018} (for $\tilde \gamma \in (0, 2 - \frac4n)$ instead of $\gamma$)
 to obtain $\Gamma=\Gamma(R,λ,γ,M_0)>0$ such that
 \begin{align*}
          z(s, t)
   &\le   \Gamma s_0^{\frac2n - 1} s
          + \Gamma s_0^{-\frac12} s^{\frac2n + \frac{\tilde \gamma}{2}} 
            \left( \int_0^{s_0} \sigma^{-\tilde \gamma} (s_0 - \sigma) w(\sigma, t) w_s(\sigma, t) \dsigma \right)^\frac12 \\
   &\le   \Gamma s_0^{\frac2n - 1} s
          + \Gamma s_0^{-\frac12} s^{\frac2n + \frac{\tilde \gamma}{2}}
            \left( \int_0^{s_0} s_0^{\gamma-\tilde\gamma}\sigma^{-\gamma} (s_0 - \sigma) w(\sigma, t) w_s(\sigma, t) \dsigma \right)^\frac12 \\ 
   &=     \Gamma s_0^{\frac2n - 1} s
          + \Gamma s_0^{-\frac12 + \frac{\gamma-\tilde \gamma}{2}} s^{\frac2n + \frac{\tilde \gamma}{2}} \sqrt{\psi(s_0, t)}
          \qquad \text{for all $s_0 \in (0, R^n)$, $t \in  (0, \min\{T, \tmax\})$.}
 \end{align*}
 Combined with (4.5) of \cite{WinklerFinitetimeBlowupLowdimensional2018}, Lemma~\ref{lm:mass_ineq}, Lemma~\ref{lm:bdd_w_psi} and Lemma~\ref{lm:beta-f},
 this shows that for any $s_0\in(0,R^n)$, $t\in(0, \min\{T, \Tmax\})$,
 \begin{align*}
        \frac{I_3(s_0,t)}{n (\gamma+1)}
  &\ge  - s_0\int_0^{s_0}s^{-γ-1} z(s,t)w(s,t)\ds  \\
  &\ge  - M_0 \Gamma \ure^{\lambda_1 t} s_0^\frac2n \int_0^{s_0}s^{-γ} \ds 
        - \Gamma s_0^{\frac12 + \frac{\gamma-\tilde \gamma}{2}} \int_0^{s_0} s^{\frac2n - 1 -\gamma + \frac{\tilde \gamma}{2}} w(s, t) \ds \cdot \sqrt{\psi(s_0, t)} \\
  &\ge  - \frac{M_0 \Gamma \ure^{\lambda_1 T}}{1-\gamma} s_0^{1+\frac2n-\gamma}
        - \Gamma s_0^{\frac12 + \frac{\gamma-\tilde \gamma}{2}} \int_0^{s_0} s^{\frac2n - 1 - \frac{\gamma - \tilde \gamma}{2}} (s_0 - s)^{-\frac12} \ds \cdot \psi(s_0, t) \\
  &\ge  - \frac{M_0 \Gamma \ure^{\lambda_1 T}}{1-\gamma} s_0^{1+\frac2n-\gamma}
        - \Gamma B\left(\frac2n - \frac{\gamma - \tilde \gamma}{2}, \frac12\right) s_0^{\frac2n} \psi(s_0, t),
 \end{align*}
 which again implies \eqref{eq:i3}.%
\end{proof}

Combining these lemmata shows that $\phi(s_0, \cdot)$ is indeed a supersolution to a superlinear ODE
as long as $s_0$ is sufficiently small and $\gamma$ can be chosen in a suitable way.

\begin{lemma}\label{lm:phi_ode_2}
Let $p \ge n$, $K \gt 0$ and $T \gt 0$.

\textbf{(i)}
  Suppose $m\in[\frac{2}{p},1+\frac{n-2}{p})$, $0\le\kappa<\frac{\alpha}{p}+\min\{\frac{n}{2p},\frac{n-2}{p}-(m-1)_+\}$ and 
  \begin{align}\label{eq:phi_ode_2:gamma-cond}
  \max\left\{0,\frac{2p\kappa}{n}-\frac{2\alpha}{n},1-\frac{2}{n}-\frac{p}{n}(m-1)_+\right\}<\gamma<\min\left\{2-\frac{4}{n}-\frac{2p}{n}(m-1)_+,1\right\}.
  \end{align}
  Then there are $C_1,C_2>0$, $θ\in(0,2)$ and $s_1\in(0,R^n)$ such that for every solution complying with \eqref{condition:upperestimate}, we have 
  \begin{align} \label{eq:phi_ode_2:ode}
        \f{\partial}{\partial t}\phi(s_0,t)
    \ge C_1 s_0^{\gamma-3} \phi^2(s_0,t)
        - C_2 s_0^{3-\gamma - θ
        }
    \qquad \text{for $t \in (0, \min\{T,\tmax\})$  and $s_0\in(0,s_1)$.}
  \end{align}
\textbf{(ii)}  
  Suppose $m\in(0,\frac{2}{p})$, $0\le\kappa<\frac{\alpha}{p}+\min\{\frac{n}{2p},\frac{n-1}{p}-\frac{m}{2}\}$ and 
  \begin{align}\label{eq:phi_ode_small-m:gamma-cond}
  \max\left\{0,\frac{2p\kappa}{n}-\frac{2\alpha}{n}\right\}<\gamma<\min\left\{2-\frac{2}{n}-\frac{pm}{n},1\right\}.
  \end{align}
  Then there are $C_1,C_2>0$ , $θ\in(0,2)$ and $s_1\in(0,R^n)$ such that \eqref{eq:phi_ode_2:ode} holds for every solution complying with \eqref{condition:upperestimate} also in this case. 
\end{lemma}
\begin{proof}
To verify part (i), we note that the prescribed conditions on $\kappa$ and $\gamma$ render Lemma~\ref{lm:i4} and the first part of Lemma~\ref{lm:i1}  applicable, in addition to Lemmata~\ref{lm:i2} and \ref{lm:i3}. Inserting their respective results into \eqref{eq:phi_ode:statement} of Lemma~\ref{lm:phi_ode}, we find that there are $c_1, c_2>0$ such that 
\begin{align*}
   \phi'(s_0,t)&\ge I_1(s_0,t)+I_2(s_0,t)+I_3(s_0,t)+I_4(s_0,t)\\
   &\ge-c_1 s_0^{\frac{3-\gamma}{2}-\frac{2}{n}-\frac{p}{n}(m-1)_+}\sqrt{\psi(s_0,t)}-c_1 s_0^{3-\frac{2}{n}-\gamma}+c_2s_0^{\gamma-3}\phi^2(s_0,t)\\
   &\pe -c_1s_0^{\f2n+1-γ} - c_1 s_0^{\f2n}ψ(s_0,t) 
   -c_1s_0^{\frac{3-\gamma}{2}+\frac{\alpha}{n}-\frac{p\kappa}{n}}\sqrt{\psi(s_0,t)}
\end{align*}
is valid for all $t \in (0,\min\{T,\tmax\})$ and for every $s_0\in(0,R^n)$. 
Young's inequality  shows that thus for every $η>0$ there is $c_3(η)>0$ satisfying 
\begin{align*}
\f{\partial}{\partial t}\phi(s_0,t)\geq& c_2 s_0^{\gamma-3}\phi^2(s_0,t)-c_1s_0^{\f2n}ψ(s_0,t)-η\psi(s_0,t)\\
&-c_3(η)\left(s_0^{3-\gamma-\frac{4}{n}-\frac{2p}{n}(m-1)_+}+s_0^{3-\gamma-\frac{2}{n}}+s_0^{1+\f2n-γ}+s_0^{3-\gamma-\frac{2p\kappa}{n}+\frac{2\alpha}{n}}\right)
\end{align*}
for all $(0,\min\{T,\tmax\})$.
If we employ Lemma~\ref{lm:phipsi}, fix $s_1$ and $η$ sufficiently small, and use that $s_0\le s_1\le R^n$, we obtain \eqref{eq:phi_ode_2:ode} with 
\[
 3-γ-θ=\min\left\{3-γ-\f4n-\f{2p}n(m-1)_+,3-γ-\f2n,1+\f2n-γ,3-γ-\f{2pκ}n+\f{2α}n\right\}.
\]
Observing that $\f{4}n+\f{2p}n(m-1)_+<\f4n+\f{2p}n\cdot\f{n-2}p=2$ due to $m \lt 1 + \frac{n-2}{p}$ 
and $\f{2pκ}n-\f{2α}n<\f{2α+n}n-\f{2α}n=1<2-\f{2}n<2$ because of $\kappa \lt \frac{\alpha}{p} + \frac{n}{2p}$, we conclude
\[
 θ=\max\left\{\f4n+\f{2p}n(m-1)_+,\f2n,2-\f{2}n,\f{2pκ}n-\f{2α}n\right\}
 \in(0,2).
\]

As to part (ii), similarly as before with Lemma~\ref{lm:i1}, part~(ii), in place of Lemma~\ref{lm:i1}, part~(i), we obtain 
\begin{align*}
  \f{\partial}{\partial t} \phi'(s_0,t)
   &\ge-c_1 s_0^{3-\gamma-\frac{2}{n}-\frac{p m}{n}}-c_1 s_0^{3-\gamma-\frac{2}{n}}+c_2s_0^{\gamma-3}\phi^2(s_0,t)  \\
   &\pe -c_1s_0^{\f2n+1-γ} - c_1 s_0^{\f2n}ψ(s_0,t)
   -c_1s_0^{\frac{3-\gamma}{2}+\frac{\alpha}{n}-\frac{p\kappa}{n}}\sqrt{\psi(s_0,t)}
\end{align*}
for all $(0,\min\{T,\tmax\})$ and $s_0\in(0,R^n)$ and conclude as before, with 
\[
 θ=\max\left\{\f2n+\f{pm}n,\f2n,2-\f2n,\f{2pκ}n-\f{2α}n\right\},
\]
where again $\f{2pκ}n-\f{2α}n<\f{2α+n}n-\f{2α}n<2$ and $\f2n+\f{pm}n<\f2n+\f2n<2$, so that $θ\in(0,2)$.
%
\end{proof}

Let us close this section by verifying that the parameter choices of Lemma~\ref{lm:phi_ode_2} are indeed feasible.

\begin{lemma}\label{lm:applicable}
 Under the conditions on $m$ and $κ$ in Lemma~\ref{lm:phi_ode_2}, part (i) or (ii), there is $γ>0$ satisfying \eqref{eq:phi_ode_2:gamma-cond} or \eqref{eq:phi_ode_small-m:gamma-cond}, respectively. Moreover, for every choice of $m$ as indicated there, the respective ranges of $κ$ are nonempty.
\end{lemma}
\begin{proof}
\textbf{(i)} 
   First we note that due to $p \ge n \gt 2$ the interval $A\defs[\frac{2}{p},1+\frac{n-2}{p})$ is not empty. For $m\in A$, due to $m<1+\frac{n-2}{p}$ and $\alpha\geq0$, we then see that $B\defs[0,\frac{\alpha}{p}+\min\{\frac{n}{2p},\frac{n-2}{p}-(m-1)_+\})$ is also not empty. Now, to check that for $m\in A$ and $\kappa\in B$ the condition \eqref{eq:phi_ode_2:gamma-cond} is not empty, we first note that $\kappa<\frac{\alpha}{p}+\frac{n-2}{p}-(m-1)_+$ and $m<1+\frac{n-2}{p}$ imply that
  \begin{align*}
  &\frac{2p\kappa}{n}-\frac{2\alpha}{n}<2-\frac{4}{n}-\frac{2p}{n}(m-1)_+\quad \text{and}\quad 1-\frac{2}{n}-\frac{p}{n}(m-1)_+<2-\frac{4}{n}-\frac{2p}{n}(m-1)_+.
  \end{align*}
  Moreover, we find that $\kappa<\frac{n}{2p}+\frac{\alpha}{p}$ and $m\lt1+\frac{n-2}{p}$ entail that
  \begin{align*}
  &\frac{2p\kappa}{n}-\frac{2\alpha}{n}<1\quad \text{and}\quad 0<1-\frac{2}{n}-\frac{p}{n}(m-1)_+<1,
  \end{align*}
  so that indeed $\max\{0,\frac{2p\kappa}{n}-\frac{2\alpha}{n},1-\frac{2}{n}-\frac{p}{n}(m-1)_+\}<\min\{2-\frac{4}{n}-\frac{2p}{n}(m-1)_+,1\}$.\\
 \textbf{(ii)} We first note that $m<\f2p$ ensures that $\f{n-1}p-\f m2\gt \f{n-2}p\gt0$, so that $κ$ with $0\le κ<\f{α}p+\min\{\f n{2p},\f{n-1}p-\f m2\}$ exists. From $κ<\f{α}p+\f{n}{2p}$ we conclude that $\f{2pκ}n-\f{2α}n< \f{2p}n(\f{α}p+\f{n}{2p})-\f{2α}n=1$, and the condition that $κ<\f{α}p+\f{n-1}p-\f m2$ shows that $\f{2pκ}n-\f{2α}n<\f{2p}n(\f{α}p+\f{n-1}p-\f m2)-\f{2α}n=2-\f2n-\f{pm}n$. Together with the fact that $2-\f2n-\f{pm}n>2-\f2n-\f2n\ge0$, this shows that  
\( \max\left\{0,\frac{2p\kappa}{n}-\frac{2\alpha}{n}\right\}<\min\left\{2-\frac{2}{n}-\frac{pm}{n},1\right\}.
 \)
\end{proof}

\section{\texorpdfstring{First conclusion: Proof of Theorem~\ref{th:main}}{First conclusion: Proof of the first main theorem}}\label{sec:proof_th11}
While Lemma~\ref{lm:phi_ode_2} and Lemma~\ref{lm:applicable} already show that $\phi(s_0, \cdot)$ is (for certain $s_0$ and $\gamma$, at least)
a supersolution to a superlinear ODE,
we still need to show that $\phi(s_0, 0)$ can be arranged to be suitably large.
We take care of this last step in the following lemma; Theorem~\ref{th:main} will then be proven directly thereafter.
\begin{lemma} \label{lm:phi_0}
  Let $\gamma \in (0, 1)$, $s_0 \in (0, R), M_1 \ge 0$ as well as $\eta \in (0, 1)$
  and set $s_{\eta} \defs (1-η) s_0$ as well as $r_1 \defs s_{\eta}^{\f1n}$.
  If
  \begin{align*}
    \int_{B_{r_1}(0)} u_0 \ge M_1,
  \end{align*}
  then
  \begin{align*}
    \phi(s_0,0) \ge \frac{η^2M_1}{\omega_{n-1}}\cdot s_0^{2-\gamma}.
  \end{align*}
\end{lemma}
\begin{proof}
We use positivity and monotonicity of $w_0\defs w(\cdot,0)$,  $w_0(s_{\eta})=\f1{\omega_{n-1}}\int_{B_{r_1}}u_0$ and $s_0-s_1=ηs_0$
as well as the fact that $1-(1-η)^{1-γ} \ge \inf_{\xi \in (0, \eta)} (1-\gamma)(1-\xi)^{-\gamma} \eta = (1-γ)η$ holds by the mean value theorem, to see that
  \begin{align*}
          \phi(s_0,0)
    &=    \intns s^{-\gamma} (s_0 - s) w_0(s) \ds 
    \ge  w_0(s_{\eta}) \int_{s_{\eta}}^{s_0} s^{-\gamma} (s_0 - s_{\eta})  \ds \\
    &\ge  \frac{η M_1}{(1-\gamma) \omega_{n-1}} s_0 \left( s_0^{1-\gamma} - s_{\eta}^{1-\gamma} \right) 
    \ge  \frac{η^2 M_1}{\omega_{n-1}} \cdot s_0^{2-\gamma}.
    \qedhere
  \end{align*}
\end{proof}

These preparations now allow us to indeed prove Theorem~\ref{th:main}.
\begin{proof}[Proof of Theorem~\ref{th:main}]
We observe that for $a,b,y_0>0$ with $\sqrt{\f ab}y_0>1$, the solution 
\[
 y(t)=\sqrt{\f ba} \cdot \f{1+\f{\sqrt{\f ab} y_0 -1}{\sqrt{\f ab} y_0 +1} \ure^{2\sqrt{ab}t}}{1-\f{\sqrt{\f ab} y_0 -1}{\sqrt{\f ab} y_0 +1} \ure^{2\sqrt{ab}t}} \quad\qquad \qquad  \text{of}\quad\qquad\qquad  
 \begin{cases}
      y' = a y^2 - b, \\
      y(0) = y_0
    \end{cases} 
\]
blows up at the finite time $t$ with $\ure^{2\sqrt{ab}t}=\f{\sqrt{\f ab} y_0 +1}{\sqrt{\f ab} y_0 -1}$. If $\sqrt{\f ab}y_0>2$, then $\f{\sqrt{\f ab} y_0 +1}{\sqrt{\f ab} y_0 -1}<\f{\sqrt{\f ab} y_0 +\f12 \sqrt{\f ab} y_0}{\sqrt{\f ab} y_0 -\f12 \sqrt{\f ab} y_0}=3$ and blow-up hence occurs before time $\f1{2\sqrt{ab}}\ln 3$. Given $m$, $n$, $p$, $κ$, $M_0$, $M_1$, $K$, $T$ from Theorem~\ref{th:main}, we use Lemma~\ref{lm:applicable} to find $γ\in(0,1)$ such that Lemma~\ref{lm:phi_ode_2} is applicable and let $C_1$ and $C_2$, $θ$ and $s_1$ be as defined there. We furthermore set $C_3\defs \f{M_1}{4ω_{n-1}}$ and introduce 
\[
 a(s_0)=C_1s_0^{γ-3},\quad 
 b(s_0)=C_2 s_0^{3-γ-θ}\quad  
\text{and} \quad
 ϕ_0(s_0)=C_3s_0^{2-γ}.
\]
By positivity of $θ$, $\sqrt{ab}\to \infty$ as $s_0\to0$; additionally, $ϕ_0^2\f{a}b\to ∞$ as $s_0\to 0$, because the exponent of $s_0$ in this expression is negative according to 
  \[
    2(2-γ)+γ-3-(3-γ-θ)
    =-2+θ<0.
  \]
We therefore can pick $s_0\in(0,s_1)$ so small that $\f{\ln 3}{2\sqrt{a(s_0)b(s_0)}}<T$ and $ϕ_0\sqrt{\f{a(s_0)}{b(s_0)}}>2$, and finally let $r_1=(\f12 s_0)^{\f1n}$. Then by Lemma~\ref{lm:phi_ode_2} in conjunction with Lemma~\ref{lm:phi_0} for $\eta=\frac12$, for every solution obeying \eqref{eq:main:cond_u0}, \eqref{eq:main:cond_mass} and \eqref{condition:upperestimate}, the function $ϕ\defs ϕ(s_0,\cdot)$ from \eqref{def:phi} satisfies 
\[
 ϕ'(t)\ge aϕ^2(t) - b \quad\text{for every } t\in(0,\min\set{T,\Tmax}),\qquad ϕ(0)\ge ϕ_0
\]
and hence $ϕ(t)\ge y(t)$ for all $t\in(0,\min\set{T,\Tmax})$, which implies $\Tmax<\f{\ln3}{2\sqrt{a(s_0)b(s_0)}}<T$ and \eqref{eq:main:blow_up}.
\end{proof}

\section{Pointwise upper estimates for \texorpdfstring{$u$}{u}: Proof of Theorem~\ref{th:ftbu_jl_pe}}
The goal of this section is to prove Theorem~\ref{th:ftbu_jl_pe}.
To that end, we first derive estimates of the form \eqref{condition:upperestimate} both for \eqref{prob:jl_ls_space} (Lemma~\ref{lm:u_pw_bdd_jl}) and \eqref{prob:pe_ls_space} (Lemma~\ref{lm:u_pw_bdd_pe})
and then apply Theorem~\ref{th:main}.

\begin{lemma}\label{lm:u_pw_bdd_jl}
  Assume that $\lambda, \mu$ not only comply with \eqref{eq:main:cond_lambda_mu} and \eqref{eq:main:cond_mu}
  but additionally satisfy $-\lambda', \mu' \ge 0$.
  Let moreover $\kappa \ge 0$, $m \gt 0$, $T, M_0 \gt 0$ and suppose that $u_0$ satisfies \eqref{eq:main:cond_u0} and $\intom u_0 = M_0$.
  Then every solution $(u, v)\in \left(C^0(\Ombar\times[0, T))\cap C^{2,1}(\Ombar\times(0, T))\right)^2$ of \eqref{prob:jl_ls_space} 
  fulfills
  \begin{align*}
    u(r, t) \le \frac{M_0 n \ure^{\lambda t}}{\omega_{n-1}} \cdot r^{-n}
    \quad \text{for all $r \in (0, R)$ and $t \in (0, T)$}.
  \end{align*}
\end{lemma}
\begin{proof}
  As in \cite[Lemma~3.7]{FuestFinitetimeBlowupTwodimensional2020} (cf.\ also \cite[Lemma~2.2]{WinklerCriticalBlowupExponent2018}),
  we employ the comparison principle to see that $u$ remains radially decreasing throughout evolution.
  To that end, we first note that by an approximation argument as in \cite[Lemma~2.2]{WinklerCriticalBlowupExponent2018},
  we may without loss of generality assume $\mu \in C^2([0, R])$ and $u_0 \in \con2$ with $\partial_\nu u_0 = 0$ on $\partial \Omega$.
  For henceforth fixed $T \in (0, \Tmax)$, these assumptions then assert that
  $u_r$ belongs to $C^0([0, R] \times [0, T)) \cap C^{2, 1}([0, R] \times (0, T))$:
  Indeed, elliptic regularity (cf.\ \cite[Theorem~1.19.1]{FriedmanPartialDifferentialEquations1976}) 
  asserts $\nabla v \in L^\infty(\Omega \times (0, T))$,
  which in conjunction with \cite[Theorem~IV.5.3]{LadyzenskajaEtAlLinearQuasilinearEquations1998} implies
  $u \in C^{1, 0}(\Ombar \times [0, T)) \cap C^{3, 1}(\Ombar \times (0, T))$.
  
  Thus, setting $f(s) \defs (s + 1)^{m-1}$ for $s \ge 0$,  and re-interpreting (JL) in the form of $u_t=∇\cdot(f(u)∇u)-∇u\cdot∇v-u(\ol M - u) +λu-μu^{1+κ}$ as an equation for the radial function $u\in C^{1,0}([0,R]\times[0,T))\cap C^{3,1}([0,R]\times(0,T))$, we may compute the radial derivative 
  \begin{align}\label{eq:u_pw_bdd_jl:urt}
          u_{rt}
    &=    \left( (f(u) u_r)_r + \frac{n-1}{r} f(u) u_r - u_r v_r + u^2 - \ol M(t) u +  \lambda u - \mu u^{1+κ} \right)_r \nn\\
    &=    (f(u) u_r)_{rr}
          + a_1 (f(u) u_r)_r
          + a_2 u_{rr}
          + b u_r
          + c
    \qquad \text{in $(0, R) \times (0, T)$},
  \end{align} 
  where
  \begin{align*}
    a_1(r, t) &\defs \frac{n-1}{r}, \qquad
    a_2(r, t) \defs - v_r(r, t), \qquad 
        c(r, t)   \defs \lambda'(r) u(r, t) - \mu'(r) u^{1+\kappa}(r, t)\quad\text{and}\\
    b(r, t)   &\defs - \frac{n-1}{r^2}f(u(r,t)) - v_{rr}(r, t) + 2 u(r, t) - \ol M(t) + \lambda(r) - (1+\kappa) \mu(r) u^\kappa(r, t)
  \end{align*}
  for $(r, t) \in (0, R) \times (0, T)$.
  Note that $-\lambda', \mu' \ge 0$ imply $c \le 0$ in $(0, R) \times (0, T)$.
  Moreover, by the second equation in \eqref{prob:jl_ls_space},
  $(r^{n-1} v_r(r, t))_r \le r^{n-1} \ol M(t)$ and hence $v_r(r, t) \le \frac{r}{n} \ol M(t)$ for $(r, t) \in (0, R) \times (0, T)$.
  This implies
 \begin{equation*}
        -v_{rr}(r, t)
    =  u(r, t) - \ol M(t) + \frac{n-1}{r} v_r(r, t)
    \le  u(r, t) - \frac1n \ol M(t)
    \le u(r, t)
    \qquad \text{for $(r, t) \in (0, R) \times (0, T)$},
  \end{equation*}
  so that by setting $c_1 \defs \sup_{(r, t) \in (0, R) \times (0, T)} (3u(r, t) + \lambda(r))$, we obtain
  \begin{align}\label{eq:u_pw_bdd_jl:b_le_c1}
    b(r, t) \le c_1
    \qquad \text{for $(r, t) \in (0, R) \times (0, T)$}.
  \end{align}
  Using that $u\in C^2([0,R]\times[0,T])$, and, in particular, that $a_1u_r=(n-1)\f{u_r}r\le (n-1)u_{rr}$ is bounded in $(0,R)\times(0,T)$, we can moreover introduce 
  \[
   c_2\defs \sup_{(r,t)\in(0,R)\times(0,T)} ((f(u)_{rr}+a_1(f(u))_r)<\infty
  \]
 and set $c_3\defs c_1+c_2+1$.

  By \eqref{eq:u_pw_bdd_jl:urt},
  since $u_r(r, t) = 0$ for $r \in \{0, R\}$ and $t \in (0, T)$ (because $u$ is radially symmetric and $(u, v)$ solves \eqref{prob:jl_ls_space})
  and as $u_{0r} \le 0$ by assumption,
  the function $y \colon [0, R] \times [0, T] \ra \R, (r, t) \mapsto u_r(r, t) - \eps \ure^{c_3 t}$
  belongs to $C^0([0, R] \times [0, T)) \cap C^{2, 1}([0, R] \times (0, T))$ and fulfills
    \begin{align}\label{eq:u_pw_bdd_jl:y_eq}
    \begin{cases}
          y_t
      \le (f(u) (y + \eps \ure^{c_3 t}))_{rr}
           + a_1 (f(u) (y + \eps \ure^{c_3 t}))_r
           + a_2 y_r
           + b y + \eps \ure^{\eps t}b
           - c_3 \eps \ure^{c_3 t}\\
\quad=(f(u) y)_{rr} 
          + a_1 (f(u) y)_r  
          + a_2 y_r
          + b y+ \eps \ure^{c_3 t}[(f(u))_{rr}+  a_1 (f(u))_r + b - c_3]\\
          \quad \le {(f(u) y)_{rr} 
          + a_1 (f(u) y)_r  
          + a_2 y_r
          + b y - \eps \ure^{c_3 t}}
           & \text{in $(0, R) \times (0, T)$}, \\
      y \lt 0 & \text{on $\{0, R\} \times (0, T)$}, \\
      y(\cdot, 0) \lt 0, & \text{in $(0, R)$}.
    \end{cases}
  \end{align}
  By the estimate for $y(\cdot,0)$ in \eqref{eq:u_pw_bdd_jl:y_eq} and continuity of $y$,
  the time $t_0 \defs \sup\{\, t \in (0, T): y \le 0 \text{ in } [0,R]\times(0, t) \,\} \in (0, T]$ is well-defined.
  Suppose $t_0 \lt T$,
  then there exists $r_0 \in [0, R]$ such that $y(r_0, t_0) = 0$ and $y(r, t) \le 0$ for all $r \in [0, R]$ and $t \in [0, t_0]$,
  hence $y_t(r_0, t_0) \ge 0$.
  As $f \ge 0$ in $[0, \infty)$, not only $y(\cdot, t_0)$ but also
  $h \colon (0, R) \ra \R, r \mapsto f(u(r, t_0)) y(r, t_0) $ attains its maximum $0$ at $r_0$.
  Since the second inequality in \eqref{eq:u_pw_bdd_jl:y_eq} asserts $r_0 \in (0, R)$,
  we conclude $h_{rr}(r_0) \le 0$, $h_r(r_0) = 0$ and $y_r(r_0, t_0) = 0$.
  However, from the first inequality in \eqref{eq:u_pw_bdd_jl:y_eq} and \eqref{eq:u_pw_bdd_jl:b_le_c1}, we could then infer the contradiction
  \begin{align*}
          0
    &\le  y_t(r_0, t_0) \\
    &\le  h_{rr}(r_0)
          + a_1(r_0, t_0) h_r(r_0)
          + a_2(r_0, t_0) y_{r}(r_0, t_0)
          + b(r_0, t_0) y(r_0, t_0)
          -\eps \ure^{c_3 t_0} \\
    &\le  - \eps \ure^{c_3 t_0}
    \lt 0,
  \end{align*}
  so that $t_0 = T$, implying $y \le 0$ in $[0, R] \times [0, T]$ and hence also $u_r \le \eps \ure^{c_3 T}$ in $[0, R] \times [0, T]$.
  Letting first $\eps \sea 0$ and then $T \nea \tmax$, this indeed gives $u_r \le 0$ in $[0, R] \times [0, \tmax)$.

This, together with Lemma~\ref{lm:mass_ineq}, implies that 
  \begin{align*}
        M_0 \ure^{\lambda_1 t}
    \ge \intom u(\cdot, t)
    =   \omega_{n-1} \int_0^R \rho^{n-1} u(\rho, t) \drho
    \ge \omega_{n-1} \int_0^r \rho^{n-1} u(r, t) \drho
    =   \frac{\omega_{n-1}}{n} r^n u(r, t) 
  \end{align*}
  for all $(r, t) \in (0, R) \times (0, \tmax)$.
\end{proof}

Arguments based on the comparison principle as used in Lemma~\ref{lm:u_pw_bdd_jl}
are apparently not expedient for the less simplified system \eqref{prob:pe_ls}.
Therefore, we rely on the pointwise upper bounds gained in \cite{FuestBlowupProfilesQuasilinear2020} instead
and argue similar as in \cite[Lemma~3.3]{WinklerFinitetimeBlowupLowdimensional2018},
whose proof instead of \cite{FuestBlowupProfilesQuasilinear2020} relies on its
predecessor \cite{WinklerBlowupProfilesLife} dealing with linear diffusion only.
\begin{lemma} \label{lm:u_pw_bdd_pe}
  Let $m \in [1, 2 - \frac2n)$, $\kappa \ge 0$, $T, L, M_0 \gt 0$
  and suppose that $\lambda, \mu$ satisfy \eqref{eq:main:cond_lambda_mu} and \eqref{eq:main:cond_mu}.
  For any $\eps \gt 0$, there is $C \gt 0$ such that with $p \defs \frac{n(n-1)}{(m-1)n + 1} + \eps$ the following holds:
  Whenever $u_0 \in \con0$ complies with \eqref{eq:main:cond_u0},
  \begin{align*}
    \intom u_0 \le M_0
    \quad \text{as well as} \quad
    u_0(x) \le L |x|^{-p} 
    \text{ for all $x \in \Omega$}
  \end{align*}
  and $(u, v)\in \left(C^0(\Ombar\times[0,T))\cap C^{2,1}(\Ombar\times(0,T))\right)^2$ is a classical solution of \eqref{prob:pe_ls_space},
  then
  \begin{align*}
    u(x, t) \le C \cdot |x|^{-p}
    \qquad \text{for all $x \in \Omega$ and $t \in (0, T)$}.
  \end{align*}
\end{lemma}
\begin{proof}
  Aiming to apply the pointwise upper bounds for quite general parabolic equations in divergence form
  obtained in \cite{FuestBlowupProfilesQuasilinear2020}, we first derive estimates for the second equation in \eqref{prob:pe_ls_space}.
  Thus, writing the second equation in \eqref{prob:pe_ls_space} in radial coordinates, we see that
  \begin{align}\label{eq:u_pw_bdd_pe:vr_est}
          r^{n-1} v_r(r, t)
    &=    \int_0^r (\rho^{n-1} v_r(\rho, t))_{ρ} \drho
     =    \int_0^r \rho^{n-1} (v(\rho, t) - u(\rho, t)) \drho
     \le  \frac{1}{\omega_{n-1}} \left( \intom v(\cdot, t) + \intom u(\cdot, t) \right)
  \end{align}
  for all $(r, t) \in (0, R) \times (0, T)$.
Since integrating the second equation in \eqref{prob:pe_ls_space} reveals that $\intom v = \intom u$ in $(0, \tmax)$,
  from \eqref{eq:u_pw_bdd_pe:vr_est} and Lemma~\ref{lm:mass_ineq} we infer that
  \begin{align*}
          r^{n-1} v_r(r, t)
    &\le  \frac{2 \ure^{\lambda_1 T} M_0}{\omega_{n-1}}
    \sfed c_1
    \qquad \text{for all $(r, t) \in (0, R) \times (0, T)$}.
  \end{align*}
  We now choose $\theta \gt n$ so large that $m-1 \gt \frac1\theta - \frac1n$
  as well as $p = \frac{n(n-1)}{(m-1)n + 1} + \eps \gt \frac{n(n-1)}{(m-1)n + 1 - \frac{n}{\theta}}$. Then
  \begin{align*}
          \intom |x|^{(n-1)\theta} |\nabla v(x, t)|^\theta \dx\le c_1^{θ}|\Omega|
     \sfed c_2
  \end{align*}
  and since $\tilde u(x, t) \defs \ure^{-\lambda_1 t} u(x, t)$, $(x, t)\in \Ombar \times [0, T)$, solves
  \begin{align*}
    \begin{cases}
      \tilde u_t \le \nabla \cdot ((\ure^{\lambda_1 t} \tilde u+1)^{m-1} \nabla \tilde u - \tilde u \nabla v) & \text{in $\Omega \times (0, T)$}, \\
      ((\ure^{\lambda_1 t} \tilde u+1)^{m-1} \nabla \tilde u - \tilde u \nabla v) \cdot \nu = 0 & \text{on $\partial \Omega \times (0, T)$}, \\
      \tilde u(\cdot, 0) = u_0 & \text{in $\Omega$},
    \end{cases}
  \end{align*}
  classically,
  an application of \cite[Theorem~1.1]{FuestBlowupProfilesQuasilinear2020}
  (with $α\defs p$, 
  $q \defs 1$,
  $K_{D, 1} \defs 1$,
  $K_{D, 2} \defs  \max\{\ure^{λ_1T(m-1)},1\}$,
  $K_S \defs 1$,
  $K_f \defs c_2$,
  $M \defs M_0$,
  $\beta \defs n-1$,
  $\mathbbmss p \defs 1$,
  $D(x, t, \tilde u) \defs (\ure^{-\lambda_1 t} \tilde u+1)^{m-1}$ and
  $S(x, t, \tilde u) \defs \tilde u$)
  yields $c_3 \gt 0$ such that $\tilde u(x, t) \le c_3 |x|^{-p}$ and hence $u(x, t) \le c_3 \ure^{\lambda_1 T} |x|^{-p}$
  for all $x \in \Omega$ and $t \in (0, T)$.
\end{proof}

With these upper estimates at hand, we can now prove Theorem~\ref{th:ftbu_jl_pe}.
\begin{proof}[Proof of Theorem~\ref{th:ftbu_jl_pe}]
  Part (i) follows directly from Theorem~\ref{th:main} and Lemma~\ref{lm:u_pw_bdd_jl}.
    
  Regarding part (ii),
  we first set $p_0 \defs \frac{n(n-1)}{(m-1)n + 1} \ge n$
  and note that the assumption $m \ge 1$ implies $m \gt \frac2n \ge \frac2{p_0}$.
  Thus, \eqref{eq:main:cond_kappa_m:1}--\eqref{eq:main:cond_kappa_m:2} reduces to \eqref{eq:main:cond_kappa_m:1}
  and if $\kappa$ satisfies \eqref{eq:ftbu_jl_pe:cond_kappa_3} then also \eqref{eq:main:cond_kappa_m:1}
  for some $p \gt p_0$ sufficiently close to $p_0$.

  For arbitrary $L \gt 0$ and $\eps \defs p - p_0 \gt 0$, we fix $C \gt 0$ as given by Lemma~\ref{lm:u_pw_bdd_pe}.
  With $r_1$ given by Theorem~\ref{th:main}, we then choose $u_0$ satisfying \eqref{eq:main:cond_u0}, having compact support in $B_{r_1}(0)$
  and fulfilling $\intom u_0 = M_0$ as well as $u_0(x) \le L |x|^{-p}$ for all $x \in \Omega$.
  Note that this is indeed possible since $p \ge n$ implies $\int_{B_{r_1}(0)} L|x|^{-p} \dx = \infty$.
  Finally, Theorem~\ref{th:main} asserts that the solution $(u, v)$ emanating from $u_0$ blows up in finite time.
\end{proof}

\small
\section*{Acknowledgments}
The authors are grateful to Y.~Tanaka for pointing out a mistake concerning the applicability of \cite[Lemma~4.8]{WinklerFinitetimeBlowupLowdimensional2018} during the proof of Lemma~\ref{lm:i3} in an earlier version of the manuscript.
The second author is partly supported by the German Academic Scholarship Foundation
and by the Deutsche Forschungsgemeinschaft within the project \emph{Emergence of structures and advantages in
cross-diffusion systems}, project number 411007140.

\footnotesize

\end{document}